\documentclass[preprint,3p,times]{elsarticle}
\usepackage{comment}

\usepackage{mathtools}
\usepackage{amsthm}
\usepackage{amssymb}
\usepackage{bm}
\usepackage{xcolor}
\usepackage{algorithm}
\usepackage{algpseudocode}
\usepackage{pgfplots}
\usepackage{pgfplotstable} 
\usepackage{tikz}

\usepackage{hyperref}
\usepackage{booktabs}
\usepackage{subcaption}

\usepackage{showkeys}

\DeclareMathOperator*{\argmin}{arg\,min}

\makeatletter
\def\plist@algorithm{Alg.\space}
\makeatother
\renewcommand\div{\operatorname{\mathrm{div}}}
\newcommand{\bfu}{\boldsymbol{u}}
\newcommand{\bfv}{\boldsymbol{v}}
\newcommand{\bfh}{\boldsymbol{h}}
\newcommand{\bfV}{\boldsymbol{V}}
\newcommand{\bfw}{\boldsymbol{w}}

\newcommand{\cT}{\mathcal T}

\newcommand{\bu}{\boldsymbol{u}}
\newcommand{\bH}{\mathbf{H}}

\newcommand{\bV}{\mathbf{V}}
\newcommand{\bff}{\boldsymbol{f}}
\newcommand{\Flin}{\mathcal{F}_\mathrm{lin}}
\newcommand{\Hdiv}{H({\div};\Omega)}
\newcommand{\onenorm}[1]{\Vert #1 \Vert_{1}}
\newcommand{\divnorm}[1]{\Vert #1 \Vert_{\mathrm{div}}}

\newtheorem{theorem}{Theorem}[section]
\newtheorem{lemma}[theorem]{Lemma}
 \newtheorem{assumption}{Assumption}
 
\theoremstyle{definition}

\theoremstyle{remark}
\newtheorem{remark}[theorem]{Remark}

\journal{arXiv}

\begin{document}

\begin{frontmatter}
\title{Least-Squares Finite Element Methods for Nonlinear Problems: A Unified Framework}

\author[label1]{Fleurianne Bertrand}
\author[label2]{Maximilian Brodbeck}
\author[label2]{Tim Ricken}
\author[label3]{Henrik Schneider}
\affiliation[label1]{organization={Technische Universität Chemnitz, Fakultät für Mathematik},
            addressline={Reichenhainer Straße 41},
            city={Chemnitz},
            postcode={09126},
            state={Sachsen},
            country={Germany}}
\affiliation[label2]{organization={Universität Stuttgart, Institut für Statik und Dynamik der Luft- und Raumfahrtkonstruktionen},
            addressline={Pfaffenwaldring 27},
            city={Stuttgart},
            postcode={70569},
            state={Baden-Württemberg},
            country={Germany}}
\affiliation[label3]{organization={Universität Duisburg-Essen,  Fakultät für Mathematik},
            addressline={ Thea-Leymann-Straße 9},
            city={Essen},
            postcode={45127},
            state={Nordrhein-Westfalen},
            country={Germany}}

\begin{abstract}
This paper presents a unified Least-Squares framework for solving nonlinear partial differential equations by recasting the governing system as a residual minimisation problem. A Least-Squares functional is formulated and the corresponding Gauss-Newton method derived, which approximates simultaneously primal and dual variables. We derive conditions under which the Least-Squares functional is coercive and continuous in an appropriate solution space, and establish convergence results while demonstrating that the functional serves as a reliable a posteriori error estimator. This inherent error estimation property is then exploited to drive adaptive mesh refinement across a variety of problems, including the stationary heat equation with either temperature-dependent or discontinuous conductivity, nonlinear elasticity based on the Saint Venant–Kirchhoff model and sea‐ice dynamics.
\end{abstract}

\begin{keyword}
Least-Squares Finite Element Method, nonlinear partial differential equations, a posteriori error estimator, adaptive mesh refinement
\end{keyword}
\end{frontmatter}

% --- Instriduction ---
\section{Introduction}
\label{sec:intro}

Nonlinear partial differential equations (PDEs) are common in modeling complex physical phenomena across science and engineering.
Owing to their rigorous theoretical foundation,  computational efficiency and ability to handle complex domains, Finite Element Methods (FEM)  have proven to be an indispensable tool for approximating PDE solutions.
Within many of these approaches, a primal variable $u$ minimizes an energy potential over a bounded domain $\Omega$ with Lipschitz boundary in $H^1(\Omega) = \{ v \in L^2(\Omega) \ : \ \Vert \nabla v \Vert_{L^2(\Omega)} <\infty \}$. 
However, these purely primal approaches often suffer from reduced robustness and fail to provide accurate approximations of certain physically relevant quantities. 
To circumvent this so-called locking phenomenon, mixed finite element methods introduce a dual quantity $\sigma$ to be approximated in $\Hdiv=\{\sigma \in L^2(\Omega)^d \ : \ \Vert \div \sigma \Vert_{L^2(\Omega)} <\infty \}$. A major challenge of these methods is the required inf-sup conditions, both at the continuous and discrete levels, to ensure well-posedness and stability. These conditions impose constraints on the choice of finite element spaces, requiring careful pairing of function spaces for primal and dual variables. Failure to satisfy the inf-sup condition can lead to unstable or inaccurate solutions, making the selection of appropriate finite element spaces a non-trivial task \cite{BB24}.

Least-Squares Finite Element Methods (LSFEM) (see e.g. \cite{BG09}) provide an attractive alternative to classical mixed methods by minimizing the residual in a Hilbert space norm. They eliminate the need for inf-sup compatible function spaces and avoid complexities associated with saddle point formulations. Another key advantage is the built-in a posteriori error estimate, which can directly be utilised in adaptive solution strategies. LSFEM has been successfully applied to a wide range of linear partial differential equations, including the Poisson problem \cite{PC94}, linear elasticity \cite{CS05, SSS11} or flow problems \cite{BC17}.

Since the resulting system of linear equations is inherently symmetric and positive definite, LSFEM allows the use of efficient linear equation solvers, making it particularly well-suited for nonlinear problems. Gauss-Newton and multilevel methods, as developed in \cite{S00}, enable the extension of LSFEM to nonlinear problems, which is crucial in applications such as nonlinear elasticity \cite{MMSW06, MSSS14, SSIS18} or general flow problems \cite{S00, M15}.
Despite these advantages, a systematic analysis of LSFEM in the nonlinear setting remains largely absent in the literature. In particular, while the Least-Squares functional in the linear case inherently provides an a posteriori error estimate, there has been limited exploration how this property can be extended to the nonlinear case. 
The purpose of this paper is to bridge this gap by developing a general framework that extends this property to broader classes of problems.

In the following, Section \ref{sec:leastsquaresapproach} introduces the Least-Squares framework for nonlinear PDEs, formulating the residual-based functional and deriving the corresponding Gauss-Newton iterative scheme. Section \ref{sec:theory} provides a detailed analysis of the theoretical properties of the Least-Squares functional, establishing conditions for coercivity, continuity, and convergence. In particular, it is proven that the Least-Squares functional serves as a reliable a posteriori error estimator in the nonlinear setting, extending its well-known properties from the linear case.
Section \ref{sec:applications} then applies the developed framework to various nonlinear PDEs, including the stationary heat equation with temperature-dependent conductivity or ReLU-activation, nonlinear elasticity using the Saint Venant–Kirchhoff model and a model for sea-ice dynamics. 

% --- The Least-Squares approach ---
\section{The Least-Squares approach}
\label{sec:leastsquaresapproach}
Let $\Omega$ be an open, bounded Lipschitz domain of $\mathbb{R}^d$ with $d\in\mathbb N_{>0}$ . Given a Hilbert space $\bV$, the nonlinear system of first-order partial differential equations under consideration reads
\begin{align}
\label{eq:residual}
\mathcal{R}(\bfu)+\bff={\boldsymbol 0}
\end{align}
for $ \bfu \in \bV$ and $\bff\in\bH$. Throughout this work, the existence of a solution 
$ \bfu^* \in \bV$ is assumed and the nonlinear functional
$ \mathcal{F}(\bu)=\|\mathcal{R}(\bu)+\bff\|_{\bH}^2 $ is minimized in a suitable Hilbert space $\bH$. 
The corresponding Fréchet derivative $\mathcal{R}^{\prime}$ of the operator $\mathcal{R}$ is given by
\begin{align*}
    \lim_{\Vert \bfh \Vert_{\bV} \rightarrow 0} 
    \frac{1}{\Vert \bfh \Vert_{\bV}}
    \Vert\mathcal{R}(\bfu+\bfh) - \mathcal{R}(\bfu) - \mathcal{R}'(\bfu)\,\bfh \Vert_{\bH} = 0\ ,
\end{align*}
allowing the computations of gradient and Hessian of the functional $\mathcal{F}$:
\begin{align*}
    &\nabla \mathcal F(\bfu)=\left(\mathcal{R}^{\prime}(\bu)\right)^*(\mathcal{R}(\bu)+\bff)\\
    &{H_\mathcal{F}(\bfu) = \left(\mathcal{R}'(\bfu) \right)^*\mathcal{R}'(\bfu) 
    + \left(\mathcal{R}''(\bfu)\right)^* (\mathcal{R}(\bfu)+\bff) } 
\end{align*}
Using these quantities, the classical Newton method in order to minimize $\mathcal F$ is expressed as
\begin{align*}
    \bfu_{n+1} = \bfu_n - \left(H_\mathcal{F}(\bfu_n)\right)^{-1} \nabla \mathcal{F}(\bfu_n) \ . 
\end{align*}
The potentially expensive evaluation of $\mathcal{R}''$ can be neglected if $\bfu_0$ is close to the solution $\bfu^*$, leading to the Gauss-Newton method. The corresponding iterative process
\begin{align}
 \label{eq:newton}
 \bfu_{n+1}=\bfu_n
 +\delta \bfu
 \;\;\text{with}\;\;
\mathcal{R}^{\prime}\left(\bfu_n\right)
\delta \bfu
=-\bff-\mathcal{R}\left(\bfu_n\right) \ ,
\end{align}
is equivalent to Newton's method since the residual $\mathcal{R}(\bfu^*)$ is zero.
To fully exploit the advantages of the Least-Squares method, in particular its inherent error estimator property, the Newton step is reformulated as a minimization problem 
\begin{align}
    \delta \bfu = \argmin_{\bfv \in \bV}\Flin(\bfv;(\bfu_n,{\bff})) 
    \quad\text{with}\quad
    \Flin(\bfv;(\bfu_n,{\bff})):=\left\|\mathcal{R}^{\prime}\left(\bfu_n\right)  \bfv
    +\bff+\mathcal{R}\left(\bfu_n\right)\right\|^2_{\bH}\ . \label{eq:Flin}
\end{align}

The Lax-Milgram lemma implies that the equivalence of the Least-Squares functional 
to the norm $\| \cdot \|_{\bV}$, i.e.
\begin{align}
\label{eq:generalEnergy}
    C_e  \|\bfw\|_\bV \leq \left\| \mathcal{R}^{\prime}\left(\bfu_n\right) \bfw\right\|_{\bH} \leq {C_c} \|\bfw\|_\bV
\quad\text{for all}\quad  \bfw\in \bV \ ,
\end{align}
is a sufficient condition for the well-posedness of the Gauss-Newton step \eqref{eq:newton}. In particular, the minimization problem \eqref{eq:Flin} has a unique minimizer.  
This result guarantees the practicability of the Gauss-Newton method outlined in Algorithm \ref{alg:gn}.

\begin{algorithm}
\caption{Generic Gauss-Newton algorithm }\label{alg:gn}
\begin{algorithmic}
\Require initial guess $\bfu_0$
\State Set $n=0$
\While{stopping criterion not met}
    \State Determine the Newton direction by solving 
    \begin{align*}
            \delta\bfu_n=\argmin_{\bfv \in \bV}
    \Flin(\bfv;(\bfu_n,{\bff}))
    \end{align*}    
    \State Set $\bfu_{n+1} = \bfu_n + \delta \bfu_n$
    \State Set $n = n+1$
\EndWhile
\end{algorithmic}
\end{algorithm}

\newpage
\noindent Since the Gauss-Newton method is equivalent to a Newton method for the residual operator 
$\mathcal R$, its established convergence theory (see \cite{Z86, D11}) is directly applicable. 
Conditions under which the Newton method is well-defined and guarantees superlinear or quadratic convergence are given in \cite{Z86, D11} summarized in the following theorem.
\begin{theorem}
\label{thm:generalGN}
Suppose $\mathcal{R} : \bV \rightarrow \bH$ and $\bff \in \bH$ and let $\bfu^*$ be a solution of $\mathcal{R}(\bfu)+ \bff=\boldsymbol{0}$.
Moreover, let the Fréchet derivative of $\mathcal{R}(\bfu)$
    exist for all $\bfu$ in an open neighborhood  of $\bfu^*$ and \eqref{eq:generalEnergy}
    holds for $\bfu_n = \bfu^*$.
    Then, there exists $\varepsilon>0$ such that the following holds: 
      \begin{itemize}
        \item[(a)] For any initial guess $\bfu_0$ in $B(\bfu^*, \varepsilon)$, the Newton method is well-defined and 
        generates a sequence $(\bfu_n)$ which converges superlinearly to $\bfu^*$.
        \item[(b)] $\bfu^*$ is the unique solution of \eqref{eq:residual} in $B(\bfu^*, \varepsilon)$. 
        \item[(c)] If additionally $\mathcal{R}'$ is Hölder continuous with exponent $\gamma$ and constant $L$ in $B(\bfu^*,\varepsilon)$ , then the convergence rate increases to ${1+\gamma}$\ .  
    In particular, if $\mathcal{R}'$ is Lipschitz continuous in $B(\bfu^*,\varepsilon)$, then the method converges quadratically. 
    \end{itemize}
\end{theorem}
\noindent 
In order to solve \eqref{eq:Flin}, the solution $\delta\bfu \in \bV$ has to fulfill the first-order necessary condition
\begin{align*}
\left(
\mathcal{R}^{\prime}
    \left(\bfu_n\right)
    \left[\delta\bfu\right], \mathcal{R}^{\prime}\left(\bfu_n\right)[\bfv]
\right)
=
\left(
    -\bff-\mathcal{R}\left(\bfu_n\right)
    , 
    \mathcal{R}^{\prime}\left(\bfu_n\right)[\bfv]
\right)
\end{align*}
for all $\bfv\in\bV$.
The energy balance \eqref{eq:generalEnergy} implies the well-posedness of the minimization of the Least-Squares functional in any finite-dimensional subspace 
$\bV_{\mathcal T}$
based on a triangulation $\mathcal T$ of $\Omega$. Up to a constant, the error of the finite element solution is bounded by the best approximation error
\begin{align*}
\Vert \bfu_{n+1} - \bfu_{n+1,\mathcal T} \Vert_\bV
&=\Vert \bfu_n + \delta\bfu - \bfu_n  - \delta\bfu_{\mathcal T} \Vert_\bV
= \Vert \delta\bfu -  \delta\bfu_{\mathcal T} \Vert_\bV \\ 
& \leq \frac 1 {C_e} 
\left\|
-\bff-
\mathcal{R}
\left(\bfu_n\right) - \mathcal{R}^{\prime}
\left(\bfu_n\right) 
\left( \delta\bfu_{\mathcal T} \right)
\right\|  = \frac 1 {C_e} \Flin ( 
    \delta\bfu_{\mathcal T})\\ &=  \frac 1 {C_e} 
    \min_{\bfv\in\bV_{\mathcal T}}
    \Flin( \bfv) = \frac 1 {C_e} 
    \min_{\bfv \in\bV_{\mathcal T}}
    \Flin ( \bfv-\delta\bfu)
    \leq \frac {C_c} {C_e}
    \min_{\bfv\in\bV_{\mathcal T}}
    \Vert \bfv-\delta \bfu \Vert \ .
\end{align*}
This leads to the inexact Newton method, outlined in Algorithm \ref{alg:ign}. 
The following convergence theorem holds under the additional assumption of a sufficiently good approximation of the Newton step, realized by a forcing sequence.
\begin{theorem} \label{thm:inexactgn}
    Let the assumptions of theorem \ref{thm:generalGN} hold and a forcing sequence $\{ \eta_n \}$ exist with $0<\eta_n <1$ 
    such that 
    \begin{align}
        \mathcal{F}_{\mathrm{lin}}(\delta \bfu_n; (\bfu_n; \bff)) \leq \eta_n \mathcal{F}(\bfu_{n+1}) \quad\text{for}\quad n\in\mathbb{N}. 
        \label{eq:forcing}
    \end{align}
    Then, there exists $\varepsilon>0$ such that it holds: 
    \begin{itemize}
        \item[(a)] For any initial guess $\bfu_0$ in $B(\bfu^*, \varepsilon)$, the inexact Newton method is well-defined and 
        generates a sequence $(\bfu_n)$ which converges to $\bfu^*$.
        \item[(b)] The sequence $(\bfu_n)$ generated by Algorithm \ref{alg:ign}, converges superlinearly.
        \item[(c)] If additionally $\mathcal{R}'$ is Lipschitz continuous on $B(\bfu^*,\varepsilon)$, then $(\bfu_n)$ generated by Algorithm \ref{alg:ign} converges quadratically.
    \end{itemize}
\end{theorem}
\begin{algorithm}
\caption{Inexact Gauss-Newton algorithm }\label{alg:ign}
\begin{algorithmic}
\Require initial guess $\bfu_0 \in \bV_{\cT}$
\State Set $n = 0$
\While{stopping criterion not met}
    \State Determine the Newton direction by solving 
    $$\delta\bfu_{n}=\argmin_{\bfv \in \bV_{\cT}} \Flin(\bfv;(\bfu_n,{\bff})) $$
    \State Set $\bfu_{n+1} = \bfu_{n} + \delta \bfu_{n}$
    \State Set $n = n+1$
\EndWhile
\end{algorithmic}
\end{algorithm}

\noindent Based on the results discussed above, the Least-Squares approach, combined with the Gauss-Newton method, offers a computationally efficient and theoretically robust framework for addressing nonlinear systems of PDEs. Subsequent sections build on these foundational results by refining the assumptions to align with the typical structure of nonlinear PDEs encountered in practical applications.

% --- A unified framework ---
\section{A unified framework}
\label{sec:theory}

The theoretical and computational framework presented so far, relies on general assumptions about the residual operator 
$\mathcal R$ and its Fréchet derivative. If the underlying system of PDEs exhibits additional structure and constraints the convergence analysis can be refined.
This section introduces new assumptions designed to leverage these structural properties.
In many practical applications, the system $\mathcal R (\bu) + \bff=\boldsymbol{0}$ typically consists of multiple coupled equations. These equations typically involve the primal variable 
$u \in U$, an additional flux 
$\sigma\in \Sigma$ and source terms $f\in H_1$, $g\in H_2$ with $\bV = U\times \Sigma$ and  $\bH = H_1\times H_2$ introduced in Section \ref{sec:intro}.
The first equation typically represents a conservation law, usually linear in $u$ and $\sigma $
\begin{align*}
    L_1(u,\sigma) + f = 0\; ,
\end{align*}
with a linear function $L_1: U\times \Sigma \rightarrow H_1$. The second equation, commonly the constitutive law, frequently exhibits a nonlinearity in $u$
\begin{align*}
    L_2(u,\sigma) + K(u) + g = 0\; ,
\end{align*}
with a linear function $L_2 : U\times \Sigma \rightarrow H_2$ and a nonlinear function $K : U \rightarrow H_2$. This can be casted into the system 
\begin{equation}
\mathcal R (\bu) = -\bff
\quad\text{with}\quad
\mathcal R (\bu) = 
\begin{pmatrix}
L_1(u,\sigma)\\
L_2(u,\sigma) + K(u)
\end{pmatrix}
\quad\text{and}\quad 
\bff = 
\begin{pmatrix}
f\\
g
\end{pmatrix}\; ,
\label{eq:fos_general}
\end{equation}
with the corresponding Least-Squares functional
\begin{align}
    \mathcal{F}(u, \sigma; f, g) = \Vert L_1(u, \sigma)+f\Vert^2_{H_1} + \Vert L_2(u, \sigma) + K(u) + g \Vert_{H_2}^2 \ .\label{eq:ls_general}
\end{align}

\noindent The nonlinear solution procedure, outlined in the previous section, is now applied to the above system.
Two additional assumptions, coercivity and continuity of the linear part of $\mathcal{F}$ as well as Litschitz continuity of the nonlinearity $K$, are formalized in the following.  
\begin{assumption}
\label{ass_linear}
  There exist positive constants $C_1, ..., C_6$ such that the following norm-equivalence holds
\begin{subequations}
\label{eq:ass_linear}
    \begin{align}
        C_1\Vert u \Vert_U^2 + C_2 \Vert\sigma\Vert_\Sigma^2 &\leq \Vert L_1(u,\sigma)\Vert^2_{H_1} + \Vert L_2(u,\sigma)\Vert^2_{H_2} 
        \label{eq:coer_general}\\
        \Vert L_1(u,\sigma)\Vert^2_{H_1} &\leq C_3 \Vert u \Vert_U^2+C_4\Vert\sigma\Vert_\Sigma^2 
        \\
        \Vert L_2(u,\sigma)\Vert^2_{H_2} &\leq C_5 \Vert u \Vert_U^2+C_6\Vert\sigma\Vert_\Sigma^2  
    \end{align}
   for all $(u, \sigma) \in U \times \Sigma$.
\end{subequations}  
\end{assumption}
\begin{assumption}
\label{ass_lipschitz}
There exists an open and non-empty set $S\subset U$ such that
\begin{align*}
    \Vert K(u) - K(v)\Vert^2_{H_2} \leq L\,\Vert u -v \Vert^2_U 
\end{align*}
holds for all $ u, v \in S$ with a Lipschitz constant $0 < L < C_1$. 
\end{assumption}
\noindent Under these assumptions, the following theorem establishes existence and uniqueness of the solution to \eqref{eq:Flin} -- equivalently, the equivalence between the norm introduced by $\mathcal{F}_{lin}$ the standard norm.
\begin{theorem}
\label{thm:normequivalence}
    Let Assumptions \ref{ass_linear} and \ref{ass_lipschitz} hold true. Then, for any $u \in S$ with $K$ Fréchet differentiable at $u$, 
    the norm-equivalence \begin{align*}
    \Vert v \Vert_U^2 + \Vert\tau\Vert_\Sigma^2 \lesssim \Vert L_1(v,\tau) \Vert^2_{H_1} + \Vert L_2(v,\tau) + K'(u)v \Vert^2_{H_2} 
    \lesssim  \Vert v \Vert_U^2 + \Vert\tau\Vert_\Sigma^2
\end{align*}
holds true for all $(v,\tau) \in U \times \Sigma$.
In particular, the Least-Squares problem \eqref{eq:Flin} has a unique solution.
\end{theorem}
\begin{proof}
    Since $K(u)$ is Fréchet differentiable, the Fréchet derivative can be bounded by the Lipschitz constant, i.e. it holds $\Vert K'(u)\Vert^2 \leq L$. 
    Triangle inequality, together with \eqref{eq:ass_linear}, implies
    \begin{align*}
        \Vert L_1(v,\tau) \Vert^2_{H_1} + \Vert L_2(v,\tau) + K'(u)v \Vert^2_{H_2} 
        &\leq \Vert L_1(v,\tau) \Vert^2_{H_1} + 2\Vert L_2(v,\tau)\Vert^2_{H_2} + 2\Vert K'(u)v \Vert^2_{H_2} \\
        &\leq (C_3 + 2C_5) \Vert v\Vert^2_U + (C_4 +2C_6) \Vert\sigma\Vert_\Sigma^2 + 2\Vert K'(u) \Vert^2 \Vert v\Vert^2_U \\
        &\leq (C_3 + 2C_5 + 2L) \Vert v\Vert^2_U + (C_4 +2C_6) \Vert\sigma\Vert_\Sigma^2 \ .
    \end{align*}
    Moreover, inequalities \eqref{eq:ass_linear} (a) - (c) implies
    \begin{align*}
        \Vert L_1(v,\tau) \Vert^2_{H_1} + \Vert L_2(v,\tau) + K'(u)v \Vert^2_{H_2} 
        &\geq \Vert L_1(v,\tau) \Vert^2_{H_1} + \Vert L_2(v,\tau)\Vert^2_{H_2} - \Vert K'(u)v \Vert^2_{H_2}\\
        &\geq C_1 \Vert v\Vert^2_U + C_2 \Vert\tau\Vert_\Sigma^2 - \Vert K'(u)\Vert^2 \Vert v \Vert_U^2 \ ,
         \end{align*}
         such that the crucial assumption $C_1-L>0$ concludes the proof with 
    \begin{align*}
        \Vert L_1(v,\tau) \Vert^2_{H_1} + \Vert L_2(v,\tau) + K'(u)v \Vert^2_{H_2} 
        &\geq (C_1-L) \Vert v\Vert^2_U + C_2 \Vert\tau\Vert_\Sigma^2 \ .
    \end{align*}   
\end{proof}

\noindent In order to guarantee convergence of the inexact Gauss-Newton (Algorithm \ref{alg:ign}) a forcing sequence \eqref{eq:forcing} is required (see Algorithm \ref{alg:cap}). Each Newton direction is then calculated using an adaptive solution strategy in order to capture local solution features at minimal computational costs, as shown in Algorithm \ref{alg:AFEM}. More details regarding the convergence of the AFEM-loop are given e.g. in \cite{BC17}.

\begin{theorem}
    Let assumptions \ref{ass_linear} and \ref{ass_lipschitz} hold true and $K$ 
    be Fréchet differentiable in $S$. Moreover, 
    let 
    $(u^*,\sigma^*) \in S\times \Sigma$
    denote a solution of $\mathcal{R}(\bfu)+\bff=0$.
    Then, the convergence of Algorithm \ref{alg:AFEM} implies that there exists $\varepsilon>0$ such that for any initial guess $\bfu_0$ in $B(\bfu^*, \varepsilon)$,
     it holds: 
    \begin{itemize}
        \item[(a)] The inexact Newton method is well-defined and 
        generates a sequence $(\bfu_n)$ which converges to $\bfu^*$. 
        \item[(b)]  $\bfu^*$ is the unique solution of $\mathcal{R}(\bfu)=0$ in $B(\bfu^*, \varepsilon)$. 
        \item[(c)] The sequence $(\bfu_n)$ generated by Algorithm \ref{alg:cap}, converges superlinearly.
        \item[(d)] If additionally $K'$ is Lipschitz continuous on $B(\bfu^*,\varepsilon)$, then $(\bfu_n)$ generated by Algorithm \ref{alg:cap} converges quadratically.
    \end{itemize}
\end{theorem}
\begin{proof}
Theorem \ref{thm:normequivalence} establishes the critical norm equivalence \eqref{eq:generalEnergy}, ensuring well-posedness of the Least-Squares functional. Assumption \ref{ass_linear} and the Fréchet differentiability of $K$ imply that $\mathcal{R}$ is Fréchet differentiable for all $\bfu \in S\times \Sigma \subset \bV$.
    With these properties in place, the conditions required to apply Theorem \ref{thm:inexactgn} are satisfied. Specifically, the inexact Newton method is well-defined and converges under the forcing sequence condition \eqref{eq:forcing}. 
\end{proof}

\begin{algorithm}
    \caption{AFEM-loop}
    \label{alg:AFEM}
    \begin{algorithmic}
    \Loop
        \State Determine $\delta\bfu_n$ by (inexactly) solving 
        $$\delta\bfu_n=\argmin_{\delta\bfu\in\bfV_\cT} \Flin(\bfv;(\bfu_n,{\bff})) $$ 
        \If {$\Flin(\delta\bfu_n; (\bfu_n, \bff)) \leq \eta_n \Vert \mathcal{R}(\bfu_n)\Vert_\bH$}
            \State stop \textbf{loop}
        \EndIf
        \State Compute error estimator $\mu_n(T) = \mathcal{F}_{\mathrm{lin}}(\delta\bfu_n; (\bfu_n, f))$ for every $T\in\mathcal{T}$
        \State Mark minimal set of elements $\mathcal{M}$ such that $\sum_{T\in\mathcal{M}}\mu_n(T) \leq \theta \sum_{T\in\mathcal{T}}\mu_n(T)$
        \State Refine all marked elements $\mathcal{M}$ of the mesh $\mathcal{T}$ and set the new mesh as $\mathcal{T}$
    \EndLoop \\
    \Return $\delta\bfu_n$
    \end{algorithmic}
\end{algorithm}
\begin{algorithm}
\caption{Inexact Gauss-Newton}\label{alg:cap}
\begin{algorithmic}
\Require initial guess $\bfu_0$
\Require $\tau \in (0,1)$
\Require marking parameter $\theta \in (0,1)$ 
\State Set $n = 0$
\State Set $\eta_n = \tau$
\While{stopping criterion not met}
    \State Determine a Newton direction $\delta\bfu_n$ by Algorithm \ref{alg:AFEM}
    \State Set $\bfu_{n+1} = \bfu_n + \delta \bfu_n$
    \State Set $n = n+1$
    \State Set $\eta_n = \min\{\tau\eta_{n-1},\Vert \mathcal{R}(\bfu_n)\Vert_\bH\}$
\EndWhile
\end{algorithmic}
\end{algorithm}

\noindent The results so far demonstrate that the combination of norm equivalence \eqref{eq:ass_linear} and Lipschitz continuity of $K$ ensures robust convergence of the inexact Gauss-Newton method. They are also fundamental to show norm-equivalenz of the nonlinear Least-Squares functional, proving that it is a reliable and efficient a posteriori error estimate. It is important to note that, unlike the previous results, the following theorem does not assume the Fréchet differentiability of $K(u)$, allowing the use of the error estimate for a broader class of nonlinear problems.
\begin{theorem}\label{thm:aposteriori_general}
Let $(u,\sigma) \in S\times \Sigma$ solve \eqref{eq:fos_general}
and the assumptions \ref{ass_linear} and \ref{ass_lipschitz} hold true. 
    Then the Least-Squares functional \eqref{eq:ls_general} is a reliable and 
    efficient error estimator, i.e. it holds
    \begin{align*}
        C_{\mathrm{rel}}\left(\Vert u-v\Vert_U^2 + \Vert \sigma-\tau\Vert_\Sigma^2\right) 
        \leq \mathcal{F}(v,\tau; f,g) 
        \leq C_{\mathrm{eff}}\left(\Vert u-v\Vert_U^2 + \Vert \sigma-\tau\Vert^2_\Sigma\right) 
    \end{align*}
    for all $(v,\tau) \in S\times \Sigma$. 
    The constants are given by
    \begin{align*}
        C_{\mathrm{rel}} &= \min\{C_1-L, C_2\} ,\\
        C_{\mathrm{eff}} &= \max\{C_3+2(C_5+L), C_4+2C_6 \} \ .
    \end{align*}
\end{theorem}
\begin{proof}
Since $(u,\sigma)$ is the solution of the minimization problem $\mathcal{F}(u, \sigma;f,g)=0$, it holds
\begin{align*}
    \mathcal{F}(v,\tau;f, g) &= \Vert L_1(v,\tau) + f \Vert^2_{H_1} + \Vert L_2(v,\tau) + K(v) + g\Vert^2_{H_2} \\
    &= \Vert L_1(u-v,\sigma-\tau) \Vert^2_{H_1} + \Vert L_2(u-v,\sigma-\tau) + K(u) - K(v) \Vert^2_{H_2}\ .
\end{align*}
On the one hand, this leads to the efficiency 
\begin{align*}
    \mathcal{F}(v,\tau;f,g) &=  \Vert L_1(v, \tau)+f\Vert^2_{H_1} + \Vert L_2(v, \tau) + K(v) + g \Vert_{H_2}^2 \\
    &= \Vert L_1(u-v, \sigma-\tau)\Vert^2_{H_1}   + \Vert L_2(u-v, \sigma-\tau) + K(u) - K(v) \Vert_{H_2}^2 \\
    &\leq \Vert L_1(u-v, \sigma-\tau)\Vert^2_{H_1} + 2 \Vert L_2(u-v, \sigma-\tau)\Vert^2_{H_2} + 2\Vert K(u) - K(v) \Vert_{H_2}^2\\
    &\leq (C_3+2(C_5+L)) \Vert u- v\Vert^2_U + (C_4+2C_6)\Vert\sigma-\tau\Vert_\Sigma^2 \ .
\end{align*}
On the other hand, inverse triangle inequality implies 
\begin{align*} 
    \mathcal{F}(v,\tau;f,g) &\geq \Vert L_1(u-v,\sigma-\tau) \Vert^2_{H_1} +  \Vert L_2(u-v,\sigma-\tau)\Vert_{H_2}^2 - \Vert K(u) - K(v) \Vert^2_{H_2} \\
    &\geq C_2\Vert\sigma-\tau\Vert_\Sigma^2 + (C_1- L) \Vert u- v\Vert_U^2
\end{align*}
and since $C_1 - L>0$ and $C_2>0$,  the reliability estimate holds true. 
\end{proof}

% --- Application of the framework ---
\section{Application of the framework}
\label{sec:applications}
In the following section, the framework will be applied to a nonlinear Poisson problem with a Lipschitz continuous coefficient or ReLU activation, the nonlinear Saint Venant–Kirchhoff model of elasticity and a model for sea-ice dynamics. Across the various applications, different boundary conditions are handled by decomposing the boundary \( \partial\Omega \) into \( \Gamma_D \) and \( \Gamma_N \), with $
\partial\Omega =\Gamma_D \cup \Gamma_N $, $\Gamma_D \cap \Gamma_N = \emptyset$ and $\mathrm{meas}(\Gamma_D)>0$ . The following usual Sobolev spaces
    \begin{align*}
    H^{1}_{\Gamma_D}(\Omega)  & = \{v \in H^{1}(\Omega; \mathbb R^d) : v  = 0  \text{ on } \Gamma_D\}, \\
    H_{\Gamma_N}(\div; \Omega) & = \{\tau \in \Hdiv : \tau \cdot \nu = 0 \text{ on } \Gamma_N  \}
    \end{align*} 
alongside with the norms $\Vert \cdot \Vert \coloneqq \Vert \cdot \Vert_{L^2(\Omega)} $, $\onenorm{\cdot} \coloneqq \Vert \cdot \Vert_{H^1(\Omega)}$ and $\divnorm{\cdot} \coloneqq \Vert \cdot \Vert_{H(\div;\Omega)}$ are used.

% --- Stationary heat equation
\subsection{Stationary heat equation with temperature-dependent thermal conductivity}
\label{sec:nlin_poisson}
As a first application of the nonlinear Least-Squares methodology the two-dimensional quasilinear stationary heat equation with temperature-dependent thermal conductivity is considered. 
For $f\in L^2(\Omega)$, $u\in H^{1}_{\Gamma_D}(\Omega)$ fulfills
    \begin{align}
    \label{eq:heat}
        -\div (\kappa(u)\nabla u) = f \ \mathrm{ in }\ \Omega, \quad
        u=0  \ \mathrm{ on }\ \Gamma_D, \quad  \nabla u \cdot \nu = 0 \ \mathrm{ on } \ \Gamma_N .
    \end{align}
    The corresponding Least-Squares functional reads
\begin{align}
    \mathcal{H}(u, \sigma; f) = 
    \omega\,\Vert f - \div\sigma\Vert^2 + \Vert \kappa(u) \nabla u  + \sigma \Vert^2 \ .\label{eq:ls_general_poisson}
\end{align}
with a weight $\omega\in \mathbb R_{>0}$.
For the analysis, note that this Least-Squares functional can be written as 
\begin{align}
    \mathcal{H}(u, \sigma; f) = 
    \omega\, 
    \Vert f - \div\sigma\Vert^2 
    + \Vert (\kappa(u)-\kappa(u^*)) \nabla u  + 
    \kappa(u^*) \nabla u+
    \sigma \Vert^2 \ ,\label{eq:ls_general_poisson_mod}
\end{align}
 i.e. exactly in the form of \eqref{eq:ls_general} with the norms
$\| \cdot \|_{H_1} = \omega\,\| \cdot \|$,
 $\| \cdot \|_{H_2}=   \| \cdot \|$, the solution space $\textbf{V} = H_{\Gamma_D}^1(\Omega) \times H_{\Gamma_N}(\div; \Omega)$ and the contributions
\begin{equation}
    L_1(u, \sigma) = -\div \sigma,\quad
    L_2(u, \sigma) = \kappa(u^*) \nabla u+\sigma\quad\text{and}\quad
    K(u) = (\kappa(u)-\kappa(u^*)) \nabla u\; .
    \label{eq:nlpoisson_definitions_framework}
\end{equation}
For the given setting, a quantitative coercivity estimate for the corresponding linear part of the functional is needed. For constant $\kappa\in\mathbb{R}_{> 0}$ the functional
\[
  \|\operatorname{div}\sigma\|^2
  \;+\;
  \|\sigma+\kappa\,\nabla u\|^2
\]
is coercive \cite{PC94}. In order to prove Assumption \ref{ass_linear} (see Theorem \ref{thm:coercivity-Linf}) more generally, the following technical Lemma is required.
\begin{lemma}
\label{lem:poissonconstants}
Let $\kappa\in \mathbb R$ and $C>0$. 
For any $\alpha \in \left(\max\left(0,\frac{C^2 - 1}{C^2 + 1}\right), 1\right)$ there exists $\beta,\gamma>0$ such that
\begin{equation*}
       1
   \;-\;
   \frac{(1-\alpha)^{2}\,\kappa^{2}}{\gamma}
>0, \quad
   \kappa^{2}
   \;-\;
   \beta
   \;-\;
   \gamma\,C^{2}
>0,  \quad
   1
   \;-\;
   \frac{\alpha^{2}\,\kappa^{2}}{\beta}
>0\ .
\end{equation*}
\end{lemma}
\begin{proof}
There exists $\beta>0$ with $\beta > \alpha^{2}\,\kappa^{2}$ and $\kappa^{2} - \gamma\,C^{2} < \beta$ which is equivalent to $\kappa^{2} - \gamma\,C^{2} > \alpha^{2}\,\kappa^{2}$. Reformulation with respect to $\gamma$ and combination with $\gamma > (1-\alpha)^2\kappa^2$ yields
\begin{equation*}
    (1-\alpha)^2\kappa^2 < \gamma < \frac{(1-\alpha)^2\kappa^2}{C^2} \quad\text{respectively}\quad g(\alpha) = 1-(1-\alpha)^{2}\,C^{2}-\alpha^{2} > 0\; .
\end{equation*}
The function $g(\alpha)$ is positive between its real roots $\alpha_0 = \frac{C^2 - 1}{C^2 + 1}$ and $\alpha_1 = 1$. For $\alpha \in \left(\max\left(0,\alpha_0 \right), 1\right)$ it is consequently possible to choose
\begin{equation*}
\beta \in 
\left({ \alpha}^2 \kappa^2 , \kappa^2 - \gamma{C^{2}} \right)
\quad\text{and}\quad 
\gamma \in \left(
(1-\alpha)^{2}\kappa^{2},
\frac{\kappa^{2}}{C^{2}}  (1  - {\alpha}^2 )\right)\ .  
\end{equation*}
\end{proof}

\begin{theorem}
\label{thm:coercivity-Linf}
Let $\kappa(u^*) \in L^\infty(\Omega)$ be  bounded from below by $\kappa_{\min}>0$ and $\omega>0$. For any $\lambda_u \in (0,\kappa_{\min}^2)$, there exist a constant $\lambda_\sigma>0$ such that
\begin{align*}
\omega\|L_1(u,\sigma)\|^{2}
+ \|L_2(u,\sigma)\|^{2}
\ge
\lambda_\sigma\,\divnorm{\sigma} + \lambda_u\,\|\nabla u\|^{2} \\
\end{align*}
holds, for all $(u,\sigma) \in H_{\Gamma_D}^1(\Omega) \times H_{\Gamma_N}(\div; \Omega).$
\end{theorem}

\begin{proof}
In order to improve readability, $\kappa(u^*)$ is following abbreviated by $\kappa$.
Estimating the Least-Squares functional with respect to $\kappa_{\min}$ from below yields
\begin{align*}
\omega \|L_1(u,\sigma)\|^{2} + \|L_2(u,\sigma)\|^{2}
&\geq \kappa_{\min}\left(\frac{\omega}{\kappa_{\min}}\|\div \sigma\|^2+\left\|\kappa^{1 / 2} \nabla u+\kappa^{-1 / 2} \sigma\right\|^2\right) \\
&= \kappa_{\min} \left(\frac{\omega}{\kappa_{\min}}\|\div \sigma\|^2  +\|\kappa^{1/2}\nabla u\|^2 +\|\kappa^{-1/2}\sigma\|^2 - 2\alpha (\nabla u, \sigma) 
+ 2(1-\alpha) (\nabla u, \sigma) \right)\; ,
\end{align*}
where $\alpha \in (0,1)$ is constant.
Integration by parts followed by an application of Young's inequality with $\gamma > 0$ and the Poincaré inequality, $C_P$ denotes the Poincaré constant, leads to
\begin{align*}
 \kappa_{\min }^{-1} \left(\omega\|L_1(u,\sigma)\|^{2}
+ \|L_2(u,\sigma)\|^{2} \right)
&\ge \left(\frac{\omega}{\kappa_{\min }}-\frac{(1-\alpha)^2}{\gamma}\right)\|\div \sigma\|^2
      +\|\kappa^{1/2}\nabla u\|^2
      +\|\kappa^{-1/2}\sigma\|^2
      -2\alpha( \nabla u,\sigma)
      -\gamma\|u\|^2 \\
&=   \left(\frac{\omega}{\kappa_{\min }}-\frac{(1-\alpha)^2}{\gamma}\right)\|\div \sigma\|^2
      +\|\kappa^{1/2}\nabla u\|^2
      +\|\kappa^{-1/2}\sigma\|^2
      -2\left(\kappa^{1/2}\nabla u,\ \alpha\,\kappa^{-1/2}\sigma \right)
      -\gamma\|u\|^2 
      \\
     & \geq
     \left(\frac{\omega}{\kappa_{\min }}-\frac{(1-\alpha)^2}{\gamma}\right)\|\div \sigma\|^2+\left(\kappa_{\min }(1-\beta)-\gamma C_P^2\right)\|\nabla u\|^2+\left(1-\frac{\alpha^2}{\beta}\right)\left\|\kappa^{-1 / 2} \sigma\right\|^2\; .
\end{align*}
Introducing $\tilde\beta=\beta\,\kappa_{\min}$, $\tilde\gamma = \gamma\,\omega$ and $\widetilde C_P^{\,2}=C_P^2/\omega$ yields
\[
\omega\|L_1(u,\sigma)\|^{2}
+ \|L_2(u,\sigma)\|^{2}\ge\
\omega\left(1-\frac{\kappa_{\min}(1-\alpha)^2}{\tilde\gamma}\right)\|\div\sigma\|^2
+
\kappa_{\min}\left(\ \kappa_{\min}-\tilde\beta-\tilde\gamma\widetilde C_P^{2}\ \right)\|\nabla u\|^2
+
\kappa_{\min}\left(\,1-\frac{\alpha^2\,\kappa_{\min}}{\tilde\beta}\right)\|\kappa^{-1/2}\sigma\|^2\; .
\]
Restricting $\alpha \in \left(\max\left(0,\frac{\widetilde C_P^{\,2} - 1}{\widetilde C_P^{\,2} + 1}\right), 1\right)$, introducing $\varepsilon > 0$ and applying Lemma \ref{lem:poissonconstants} leads to
\[
\tilde\beta=\alpha^2\kappa_{\min}+\varepsilon,\qquad
\tilde\gamma=(1-\alpha)^2\kappa_{\min}+\varepsilon\; .
\]
The coefficient in front of \(\|\nabla u\|^2\)
\[
\begin{aligned}
\kappa_{\min}\left(\kappa_{\min}-\tilde\beta-\tilde\gamma\,\widetilde C_P^{2}\right)
&=\kappa_{\min}\left(\kappa_{\min}\left(1-\alpha^2-(1-\alpha)^2\widetilde C_P^{2}\right)
-\varepsilon\left(1+\widetilde C_P^{2}\right)\right)
\end{aligned}
\]
tends to $\kappa_{\min}^2 g_\omega(\alpha)$
with
\[
g_\omega(\alpha):=1-\alpha^2-(1-\alpha)^2\,\widetilde C_P^{\,2}
=1-\alpha^2-(1-\alpha)^2\frac{C_P^2}{\omega}.
\]
when $\varepsilon\to0^+ $.
Maximizing $g_\omega$ over $\alpha$ yields
\[
\alpha_{\mathrm{opt}}(\omega)=\frac{\widetilde C_P^{\,2}}{1+\widetilde C_P^{\,2}}
=\frac{C_P^2}{C_P^2+\omega}
\quad\text{and}\quad
g_\omega\left(\alpha_{\mathrm{opt}}(\omega)\right)
=\frac{1}{1+\widetilde C_P^{\,2}}
=\frac{\omega}{C_P^2+\omega}.
\]
Consequently, for a given $\omega>0$, the coercivity constant $\lambda_u$ can be chosen
arbitrarily close to
\[
{\ \frac{\kappa_{\min}^2\,\omega}{C_P^2+\omega}\ }.
\]
This function of $\omega$ is strictly increasing, since
\[
\frac{\mathrm{d}}{\mathrm{d}\omega}\frac{\kappa_{\min}^2\,\omega}{C_P^2+\omega}
=\frac{\kappa_{\min}^2\,C_P^2}{(C_P^2+\omega)^2}>0,
\]
and satisfies
\[
\lim_{\omega\to\infty}\frac{\kappa_{\min}^2\,\omega}{C_P^2+\omega}
=\kappa_{\min}^2 \quad\text{and}\quad
\lim_{\omega\to0^+}\frac{\kappa_{\min}^2\,\omega}{C_P^2+\omega}=0.
\]
In particular, by increasing \(\omega\), the coercivity constant $\lambda_u$ can be made arbitrarily 
close to \(\kappa_{\min}^2\), while the coefficients of the 
\(\|\operatorname{div}\sigma\|^2\) and \(\|\sigma\|^2\) terms remain strictly positive.
The assumptions on $\kappa$ implies $\|\sigma \| \approx \| \kappa^{-1/2}\sigma \|$ which concludes the proof.
\end{proof}
\noindent In addition to norm equivalence of the linear part of the Least-Squares functional, 
it remains to prove, that the nonlinearity $K(u)$ is Lipschitz continuous with a Lipschitz-constant smaller than $ \kappa_{\min}^2$.
\begin{assumption}
\label{ass:thermo}
$ $\\\vspace{-\baselineskip}
\begin{enumerate} 
\item $\Omega$ is a polygonal domain and  $\omega_{\max}$ is the measure of its largest interior angle of $\Omega$. Moreover we set
\[
s_\Omega\ :=\ \min\Big\{\,1,\ \frac{\pi}{\omega_{\max}}\,\Big\}\in(0,1].
\]
\item The coefficient
$
  \kappa : \mathbb R \rightarrow \mathbb R
$
globally Lipschitz; that is, there exist 
 constants \(L_\kappa>0\) and \(\kappa_{\min}>0\) such that
    $
      \left|\kappa(t_1) - \kappa(t_2)\right|
      \;\le\; L_\kappa\,|t_1-t_2|
      $
and
    $
\kappa(t_1)\;\ge\;\kappa_{\min}  \ \forall\,t_1,t_2\in\mathbb R
    $ .
\end{enumerate}
\end{assumption}

\begin{remark}
\label{rem:assthermo} 
Assumption \ref{ass:thermo} implies that the weak solution 
$u^*$ belongs to $H^{1+s}(\Omega)$ for every $0<s<s_{\Omega}$ \cite{L61}. Moreover, standard energy arguments yield 
\begin{subequations}
\begin{align}
\label{eq:poisson_reg_L2}
\|\nabla u^*\| &\leq \frac{C_P}{ \kappa_{\min }}\|f\|\\ 
\label{eq:poisson_reg_Lp}
\|\nabla u^*\|_{L^p(\Omega)} &\leq \frac{C_{M}}{\kappa_{\min }} \|f\|
\end{align}
\end{subequations}
where $C_\Omega$ denotes the geometric constant depending on the Lipschitz character of the domain, see e.g. \cite{D06,G11}, and the constant
$C_{M} = \left(1+\frac{1}{s_{\Omega}-s}\right) 
C_{H^{1+s}\ \hookrightarrow\ W^{1,p}}
 C_{\Omega}
$.
This estimate aligns with the framework of Meyers' theorem in \cite{Meyers}, which guarantees the existence of some $p>2$ such that $\nabla u^* \in L^{p}(\Omega)$ whenever $f \in L^2(\Omega)$. In our context, this means that the theory remains valid under the weaker Meyers regularity, even though the corresponding constant $C_M$ would be less explicit.
\end{remark}
 
\begin{theorem}
\label{lem:ass_lipschitz_poissonNL_2D}
Let $u^*\in H^{1}_0(\Omega)  \cap H^{1+s}(\Omega)$ solve \eqref{eq:heat} under Assumption~\ref{ass:thermo}, 
with the notation of Remark~\ref{rem:assthermo}. Moreover, let the source 
 \(f\in L^2(\Omega)\) 
satisfy
\begin{equation*}
\|f\| < \frac{\kappa_{\min }^2 }{L_\kappa C_{H^1 \hookrightarrow L^q}C_P
 C_M}
\end{equation*}
where $q = \frac{2p}{p-2}$ denotes the Hölder conjugate of $p$.
Then there exists $\delta$ and a constant $L\in(0,\kappa_{\min})$ such that
\[
\|K(u)-K(v)\|
\leq
L\,\|\nabla(u-v)\|,
\]
for all $u,v\in  
B_\delta:=\left\{\,w\in H^{1}_0(\Omega)  \cap H^{1+s}(\Omega):\ \|w-u^*\|_{{1+s}}\le \delta\,\right\}$ .
\end{theorem}
\begin{proof}
The definition of $K$ in \eqref{eq:nlpoisson_definitions_framework} yields
\begin{equation}
K(u)-K(v)=(\kappa(u)-\kappa(u^*))\,\nabla(u-v)+(\kappa(u)-\kappa(v))\,\nabla v \ .
\label{eq:nlpoisson-proof-lipschitz-step1}
\end{equation}
By the global Lipschitz property of $\kappa$ and the embedding
$H^{1+s}\hookrightarrow L^\infty$ the first term of \eqref{eq:nlpoisson-proof-lipschitz-step1} can be bounded by 
\[
\|(\kappa(u)-\kappa(u^*))\,\nabla(u-v)\| \le L_\kappa\,\|u-u^*\|_{L^\infty(\Omega)}\,\|\nabla(u-v)\|
\le L_\kappa\,C_{H^{1+s}\hookrightarrow L^\infty}\,\delta\,\|\nabla(u-v)\|.
\]
%\emph{Estimate of $T_2$.} 
Applying Hölders inequality and the embeddings
$H^{1}\hookrightarrow L^{q}$ and $H^{1+s}\hookrightarrow W^{1,p}$ in the second term of \eqref{eq:nlpoisson-proof-lipschitz-step1} yields
\[
\|(\kappa(u)-\kappa(v))\,\nabla v\|
\le L_\kappa\,\|u-v\|_{L^{q}(\Omega)}\,\|\nabla v\|_{L^{p}(\Omega)}
\le L_\kappa\,C_{H^{1}\hookrightarrow L^{q}}C_P\,\|\nabla(u-v)\|
      \Big(\|\nabla u^*\|_{L^{p}(\Omega)}+C_{H^{1+s}\hookrightarrow W^{1,p}}\,\delta\Big).
\]
and consequently 
\begin{align*}
\|K(u)-K(v)\|_{L^2(\Omega)} \le L(\delta)\,\|\nabla(u-v)\|
\end{align*}
with
\begin{align*}
L(\delta)
&= L_\kappa
C_{H^{1}\hookrightarrow L^{q}}C_P\,\|\nabla u^*\|_{L^{p}}
+
L_\kappa
\delta
\left(
C_{H^{1+s}\hookrightarrow L^\infty}C_P
+ C_{H^{1}\hookrightarrow L^{q}}C_P\,C_{H^{1+s}\hookrightarrow W^{1,p}}
\right)\; .
\end{align*}
The energy estimate \eqref{eq:poisson_reg_Lp} together with Assumption \ref{ass:thermo} implies
\[
L(0)
\le
L_\kappa\,
C_{H^{1}\hookrightarrow L^{q}}\,
C_P\,
\frac{C_M}{\kappa_{\min}}\,\|f\|
< \kappa_{\min}^2.
\]

\noindent Since $L(\delta)$ is continuous, there exists a $\delta >0$ with $L(\delta)\in (0,\kappa_{\min})$
such that
\[
\|K(u)-K(v)\|\leq L(\delta)\,\|\nabla(u-v)\|
\quad\text{for all}\quad\,u,v\in B_\delta.
\]
\end{proof}

The previous Theorems \ref{thm:coercivity-Linf} and \ref{lem:ass_lipschitz_poissonNL_2D} imply that assumptions \ref{ass_linear} and \ref{ass_lipschitz} hold. Consequently, Theorems \ref{thm:normequivalence} and \ref{thm:aposteriori_general} apply, which means that the Least-Squares functional \eqref{eq:ls_general_poisson} serves as an error indicator for the nonlinear, quasi-stationary heat equation. Corresponding numerical results demonstrate this for two domains, a rectangle, as well as an L-shape made of silicon with variable thermal conductivity 
\begin{equation*}
    \kappa(u) = 5\,\left(\cos(0.38\,u + 2.45) + 1 \right) \quad\text{with}\quad u = \frac{2\,\theta}{800\mathrm{K}} - \frac{3}{2}\quad\text{and}\quad u\in[-1.0,\;1.0]\; ,
\end{equation*}
based on the measurements of \cite{M13}. Convergence of numerical solutions is measured with respect to:
\begin{equation}
    \vert\vert\vert \textbf{u}-\textbf{u}_\cT \vert\vert\vert^2  = \|\nabla\left(u-u_\cT\right)\|^{2} + \Vert\sigma-\sigma_\cT\Vert_{\div}^2\; .
    \label{eq:error_norm_poisson}
\end{equation}

\textbf{Example 1:}\\
At first, a unit square with a manufactured solution
\begin{equation*}
    u = \sin(2\pi\,x)\,\sin(2\pi\,y) \quad\text{and}\quad f = -\div \left(\kappa(u)\,\nabla u\right)
\end{equation*}
is considered.
Homogeneous Dirichlet boundary conditions are applied on the entire boundary $\partial\Omega$.
The experimental order of convergence (e.o.c) of the error, as well as the Least-Squares functional \eqref{eq:ls_general_poisson} under uniform mesh refinement, are reported in Table \ref{tab:PoissonNL-ManSol}. 
Optimal e.o.c for the error \eqref{eq:error_norm_poisson} as well as the Least-Squares functional are achieved for first and second order approximations. The efficiency of the error estimate is close to one, the Gauss-Newton procedure requires 6-13 iterations to converge.
\begin{table}[h]
\centering
\begin{tabular}{@{}l|lllll|lllll@{}}
\toprule
     & \multicolumn{5}{c}{$k = 1$} & \multicolumn{5}{c}{$k = 2$} \\ 
$h$ & $\vert\vert\vert \textbf{u}-\textbf{u}_\cT \vert\vert\vert$ & & $\mathcal{H}^{1/2}$ & & $i_{eff}$ & $\vert\vert\vert \textbf{u}-\textbf{u}_\cT \vert\vert\vert$ & & $\mathcal{H}^{1/2}$ & & $i_{eff}$ \\ \midrule
1/2  & $4.86e+1$ & $-$    & $4.53e+1$ & $-$    & $0.93$ & $4.37e+1$ & $-$    & $3.99e+1$ & $-$    & $0.91$ \\
1/4  & $4.54e+1$ & $0.10$ & $4.53e+1$ & $0.00$ & $0.99$ & $1.31e+1$ & $1.74$ & $1.30e+1$ & $1.61$ & $0.99$ \\
1/8  & $2.10e+1$ & $1.11$ & $2.10e+1$ & $1.11$ & $1.00$ & $2.97e+0$ & $2.14$ & $2.97e+0$ & $2.13$ & $1.00$ \\  
1/16 & $9.89e+0$ & $1.09$ & $9.91e+0$ & $1.09$ & $1.00$ & $1.15e+0$ & $1.36$ & $1.15e+0$ & $1.36$ & $1.00$ \\
1/32 & $4.99e+0$ & $0.99$ & $5.00e+0$ & $0.99$ & $1.00$ & $2.92e-1$ & $1.98$ & $2.92e-1$ & $1.98$ & $1.00$ \\         
1/64 & $2.50e+0$ & $0.99$ & $2.50e+0$ & $1.00$ & $1.00$ & $7.33e-2$ & $1.99$ & $7.34e-2$ & $2.00$ & $1.00$ \\   \bottomrule
\end{tabular}
\caption{Convergence history for a manufactured solution of the stationary heat equation with temperature dependent conductivity: Values for the error $\vert\vert\vert \mathbf{u} - \mathbf{u}_\cT \vert\vert\vert$ following \eqref{eq:error_norm_poisson}, the Least-Squares functional $\mathcal{H}$ and the effectivity index $i_{eff}$ using $u_\cT \in \left(\mathrm{P}_k\right)^2$ and $\sigma_\cT \in \left(\mathrm{RT}_{k-1}\right)^2$ with $k\in\{1,\, 2\}$.}
\label{tab:PoissonNL-ManSol}
\end{table}

\textbf{Example 2:}\\
In order to highlight the use of the Least-Squares functional as an error indicator, the second example considers an L-shaped domain $\Omega=[-1,\,1]^2 \backslash [0,1]$ with mixed boundary conditions 
\begin{align*}
u &= -0.8 & &\text{on}\;\left\{(x,y)\in\partial\Omega\, : \,x=1\right\}\\
u &= 0.8 & &\text{on}\;\left\{(x,y)\in\partial\Omega\, : \,y=1\right\}\\
\nabla u\cdot\nu &= 0.0 & &\text{on}\;\left\{(x,y)\in\partial\Omega\, : \,x \neq 1 \land y \neq 1\right\}\; .
\end{align*}
A spatially constant source term $f = -0.05$ is applied.
Due to the singularity in the reentrant corner, adaptive mesh refinement is applied, using a Dörfler marking strategy with a bulk parameter of 0.5 for first- respectively 0.8 for second order approximations. 
The resulting mesh after six refinement cycles as well as the dimensionless temperature field are shown in Figure \ref{fig:PoissonNL-LShape}.
\begin{figure}[h]
    \centering
    \begin{subfigure}{0.45\textwidth}
        \centering
        \includegraphics[scale=0.3]{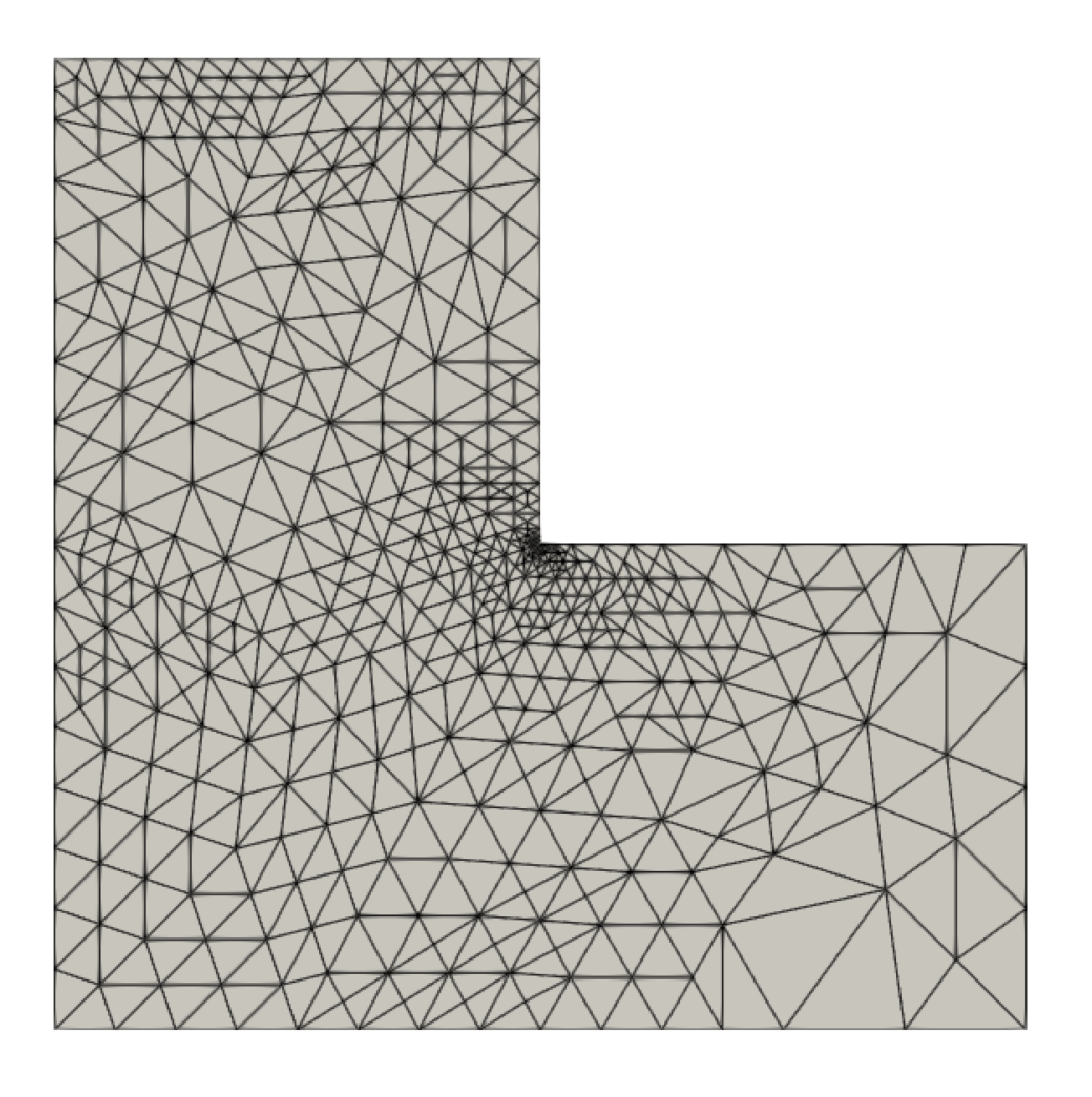}
        \caption{}
    \end{subfigure}
    \begin{subfigure}{0.45\textwidth}
        \centering
        \includegraphics[scale=0.3]{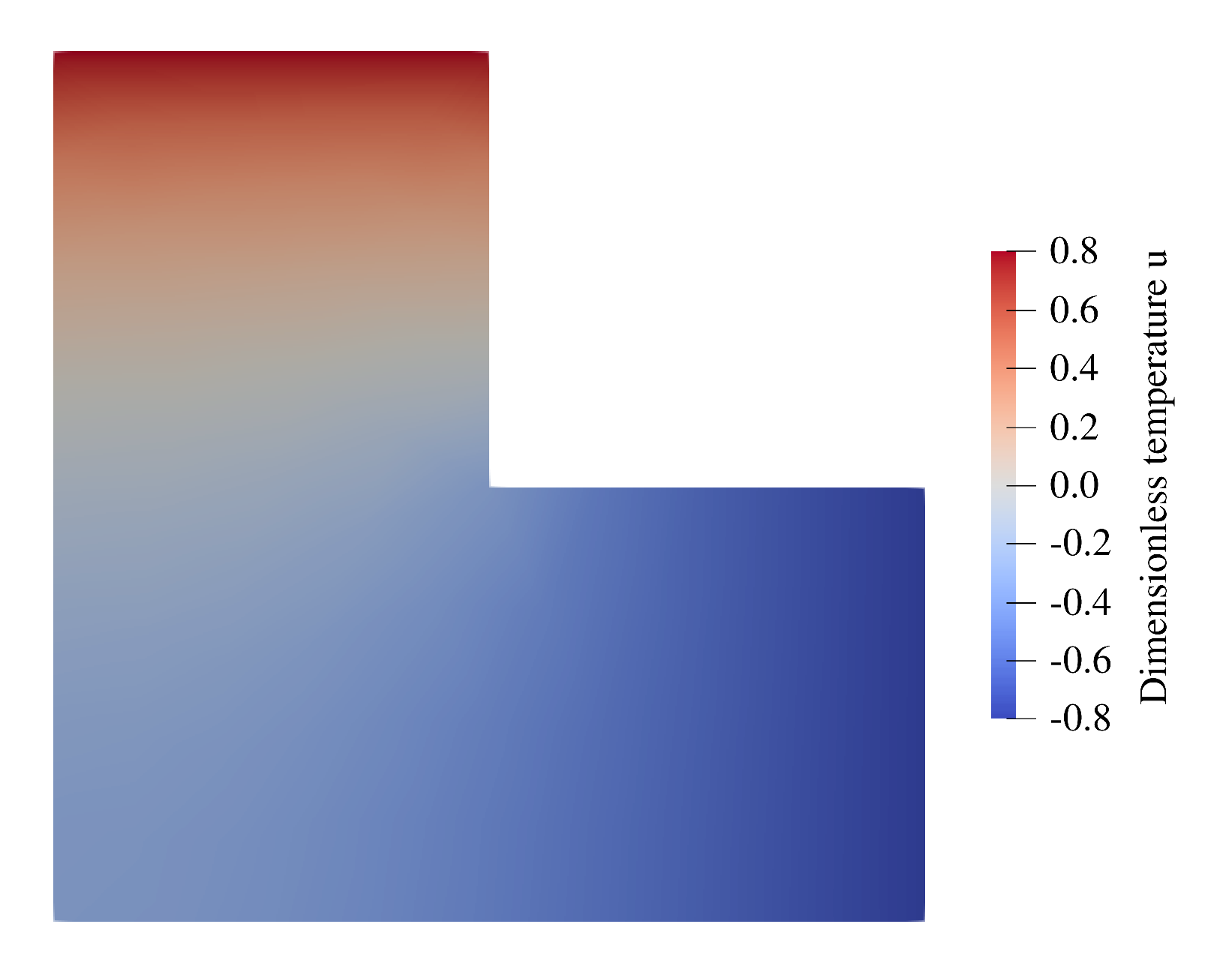}
        \caption{}
    \end{subfigure}
    \caption{Solution of the stationary heat equation on an L-shaped domain using $u_\cT \in \left(\mathrm{P}_1\right)^2$ and $\sigma_\cT \in \left(\mathrm{RT}_{0}\right)^2$: The adapted mesh after six refinement cycles (a) and the corresponding solution (b).}
    \label{fig:PoissonNL-LShape}
\end{figure}

\noindent The convergence history of two Least-Squares methods is detailed in Table \ref{tab:PoissonNL-LShape}.
Optimal e.o.c. are achieved, while the effectivity indices $i_{eff}$ of the underlying error estimators are very close to one.
\begin{table}[h]
\centering
\begin{tabular}{@{}ll|llllll@{}}
\toprule
$k$ & $n$ & $n_{DOF}$ & $\vert\vert\vert \textbf{u}-\textbf{u}_\cT \vert\vert\vert$ & & $\mathcal{H}^{1/2}$ & & $i_{eff}$ \\ \midrule
1 & 1  & $81$    & $3.72e-1$ & $-$    & $3.77e-1$ & $-$    & $1.01$ \\
1 & 2  & $131$   & $3.06e-1$ & $0.40$ & $2.87e-1$ & $0.57$ & $0.94$ \\
1 & 3  & $221$   & $2.48e-1$ & $0.40$ & $2.32e-1$ & $0.41$ & $0.93$ \\
1 & 4  & $359$   & $2.00e-1$ & $0.45$ & $1.75e-1$ & $0.58$ & $0.88$ \\
1 & 9  & $7124$  & $3.96e-2$ & $0.48$ & $3.67e-2$ & $0.50$ & $0.93$ \\
1 & 10 & $13661$ & $2.86e-2$ & $0.50$ & $2.66e-2$ & $0.49$ & $0.93$ \\ \midrule
2 & 1  & $257$   & $1.41e-1$ & $-$    & $1.58e-1$ & $-$    & $1.12$ \\
2 & 2  & $443$   & $9.81e-2$ & $0.79$ & $1.02e-1$ & $0.79$ & $1.04$ \\
2 & 3  & $643$   & $7.44e-2$ & $1.07$ & $6.88e-2$ & $1.07$ & $0.93$ \\
2 & 4  & $968$   & $5.55e-2$ & $0.95$ & $4.67e-2$ & $0.95$ & $0.84$ \\
2 & 9  & $8651$  & $4.38e-3$ & $0.98$ & $4.75e-3$ & $0.98$ & $1.08$ \\
2 & 10 & $14027$ & $2.79e-3$ & $0.94$ & $3.01e-3$ & $0.94$ & $1.08$ \\ \bottomrule
\end{tabular}
\caption{Convergence history of the stationary heat equation with temperature dependent conductivity on an L-shaped domain: Values for the error $\vert\vert\vert \mathbf{u} - \mathbf{u}_\cT \vert\vert\vert$ following \eqref{eq:error_norm_poisson}, the Least-Squares functional $\mathcal{H}$ and the effectivity index $i_{eff}$ using $u_\cT \in \left(\mathrm{P}_k\right)^2$ and $\sigma_\cT \in \left(\mathrm{RT}_{k-1}\right)^2$ with $k\in\{1,\, 2\}$.}
\label{tab:PoissonNL-LShape}
\end{table}

\subsection{Remarks on discontinuous coefficients}
For completeness, this section remarks that extending the analysis to the case of discontinuous coefficients poses no additional difficulties. To this end, let $\Omega$ be partitioned into two subdomains $\Omega_1$ and $\Omega_2$, on each of which the parameter is constant
\[
\kappa(x)=
\begin{cases}
\kappa_1, & x\in\Omega_1\\
\kappa_2, & x\in\Omega_2
\end{cases}\; .
\]
Then, on each subdomain $\Omega_i$, integrating by parts yields
\[
\begin{aligned}
\Vert L_1(u, \sigma) \Vert^2 + \Vert L_2(u,\sigma)\Vert^2 = &\|\operatorname{div}\,\sigma\|_{L^2(\Omega_i)}^2 
+\kappa_i^2\,\|\nabla u\|_{L^2(\Omega_i)}^2 
+\|\sigma\|_{L^2(\Omega_i)}^2
\\&-2\alpha\,\kappa_i\,\left( \nabla u,\sigma\right)_{\Omega_i}
+2(1-\alpha)\,\kappa_i\,\left( u,\div\sigma\right)_{\Omega_i}
-2(1-\alpha)\,\kappa_i \langle u, (\sigma\cdot n)\rangle_{\partial\Omega_i}
\end{aligned}
\]
For the given boundary conditions, the only non-vanishing boundary contributions arise from the internal interface 
$
\Gamma=\partial\Omega_1\cap\partial\Omega_2$ and the interface term can be estimated using the duality between $H^{1/2}(\Gamma)$ and $H^{-1/2}(\Gamma)$. Indeed
\[
\langle u,\sigma\cdot n\rangle_\Gamma\le \|u\|_{H^{1/2}(\Gamma)}\,\|\sigma\cdot n\|_{H^{-1/2}(\Gamma)},
\]
and the Young's inequality with parameter $\delta>0$ yields
\[
-2(1-\alpha)(\kappa_1-\kappa_2)\langle u,\sigma\cdot n\rangle_\Gamma
\ge -\delta\,\|u\|_{H^{1/2}(\Gamma)}^2 - \frac{(1-\alpha)^2(\kappa_1-\kappa_2)^2}{\delta}\,\|\sigma\cdot n\|_{H^{-1/2}(\Gamma)}^2.
\]
Moreover, combining the standard trace inequality with the Poincaré–Friedrichs inequality leads to
\begin{equation*}
\|u\|_{H^{1/2}(\Gamma)}^2\le C_{\mathrm{tr},u}^2\left(\|u\|^2+\|\nabla u\|^2\right)
\le C_{\mathrm{tr},u}^2\,(1+C_P^2)\,\|\nabla u\|^2
\quad\text{and}\quad
\|\sigma\cdot n\|_{H^{-1/2}(\Gamma)}^2\le C_{\mathrm{tr},\sigma}^2\left(\|\sigma\|^2+\|\operatorname{div}\,\sigma\|^2\right).    
\end{equation*}
Thus, the coercivity estimate reads
\[
\begin{aligned}
\Vert L_1(u, \sigma) \Vert^2 + \Vert L_2(u,\sigma)\Vert^2 \ge\; &\sum_{i=1}^{2}\Biggl\{
\Biggl[\left(1-\frac{(1-\alpha)^2\,\kappa_i^2}{\gamma_i}\right)
-\frac{(1-\alpha)^2(\kappa_1-\kappa_2)^2}{\delta}\,C_{\mathrm{tr},\sigma}^2\Biggr]
\|\operatorname{div}\,\sigma\|_{L^2(\Omega_i)}^2\\[1mm]
&\quad+\Biggl[\left(1-\frac{\alpha^2\,\kappa_i^2}{\beta_i}\right)
-\frac{(1-\alpha)^2(\kappa_1-\kappa_2)^2}{\delta}\,C_{\mathrm{tr},\sigma}^2\Biggr]
\|\sigma\|_{L^2(\Omega_i)}^2\\[1mm]
&\quad+
\left[\kappa_i^2-\beta_i-\gamma_i\,C_P^2 - \delta\,C_{\mathrm{tr},u}^2\,(1+C_P^2)\right]
\|\nabla u\|_{L^2(\Omega_i)}^2
\Biggr\}.
\end{aligned}
\]
In a first step, choose
\begin{equation*}
\delta=\theta \frac{\min _{i=1,2}\left\{\kappa_i^2-\beta_i-\gamma_i C_P^2\right\}}{C_{\mathrm{tr}, u}^2\left(1+C_P^2\right)}
\ .
\end{equation*}
According to Lemma \ref{lem:poissonconstants}  $1-\alpha$ can now be chosen small enough to ensure that all the terms remains positive. 

\subsection{ReLU Activation}
Equations involving ReLU-type nonlinearities appear frequently in control theory and physics-informed machine learning. A representative example is
\[
-\Delta u + \max(u,0) = f \quad \text{in } \Omega \ .
\]
To cast this second‐order PDE into the presented framework, a flux variable
$
\sigma = \nabla u
$ is introduced. 
Consequently, the resulting Least-Squares function yields
\begin{equation}
    \mathcal{G}(u, \sigma; f ) =  \|f - \div\sigma + \max(u,0)\|^2 + \|\sigma - \nabla u\|^2 \,,
    \label{eq:ls_general_relu}
\end{equation}
where coherence with the abstract functional \eqref{eq:ls_general} is achieved for the norms $\| \cdot\|_{H_1} = \|\cdot\|_{H_2} = \|\cdot\|$, the function space $\textbf{V} = H_{\Gamma_D}^1(\Omega) \times H_{\Gamma_N}(\div; \Omega)$ and the definitions
\begin{align*}
L_1(u,\sigma) = \sigma - \nabla u,\quad
L_2(u,\sigma) = f - \div\sigma + u
\quad\text{and}\quad
K(u) = \max(u,0) - u\; .
\end{align*} 

Before dealing with the nonlinearity, coercivity of the linear part of $\mathcal{G}$ is proven in the following theorem.
\begin{theorem}
    The linear part of the Least-Squares functional $\mathcal{G}$ is coercive, i.e.
        \begin{equation*}
            \left(1 + \frac{1}{2 C_P^2}\right)\,\Vert u\Vert^2 + \frac{1}{2}\Vert\nabla u\Vert^2 + \Vert\sigma\Vert_{\div} \leq \Vert L_1(u,\sigma)\Vert^2 + \Vert L_2(u,\sigma)\Vert^2
        \end{equation*}
    for all $(u,\sigma) \in H_{\Gamma_D}^1(\Omega) \times H_{\Gamma_N}(\div; \Omega)$.
\end{theorem}
\begin{proof}
    Integration by parts and the Poincaré inequality lead to proposition
    \begin{align*}
        \Vert L_1(u,\sigma)\Vert^2 + \Vert L_2(u,\sigma)\Vert^2 &= \Vert \sigma\Vert^2 + \Vert \nabla u\Vert^2 +\Vert u\Vert^2 +\Vert\div\sigma\Vert^2 \\
        &\geq \frac{1}{2} \Vert \nabla u\Vert^2 + \left( 1+ \frac{1}{2C_P^2} \right)   \Vert u\Vert^2 + \Vert\div\sigma\Vert_{\div}^2 \ .    \end{align*}
\end{proof}
\noindent The nonlinearity $K$ is Lipschitz continuous with $L=1$, i.e.
\begin{align}
    \Vert K(u) - K(v) \Vert \leq \Vert u -v \Vert
\end{align}
for all $u,v\in L^2(\Omega)$, but not Fréchet differential at $0$. 
Therefore, only Theorem \ref{thm:aposteriori_general} applies and Newtons method has to be replaced by more general solution procedures, e.g. the primal-dual active set strategy \cite{HIK}.
Assumption \ref{ass_lipschitz} usually requires the $\Vert \cdot \Vert_U$ norm, but can be weakened to only the $L^2$ part of the norm. The proof of Theorem \ref{thm:aposteriori_general} can be adapted.
As Assumptions \ref{ass_linear} and \ref{ass_lipschitz} hold, the Least-Squares functional $\mathcal{G}$ is, following Theorem 
\ref{thm:aposteriori_general}, a reliable error estimator.
This is numerically illustrated on the basis of the previously used L-shape $\Omega=[-1,\,1]^2 \backslash [0,1]^2$, with modified boundary conditions 
    \begin{align*}
        u &= -2.0 & &\text{on}\;\left\{(x,y)\in\partial\Omega\, : \,x=1\right\}\\
        u &= \textcolor{white}{-}3.0 & &\text{on}\;\left\{(x,y)\in\partial\Omega\, : \,y=1\right\}\\
        \nabla u\cdot\nu &= \textcolor{white}{-}0.0 & &\text{on}\;\left\{(x,y)\in\partial\Omega\, : \,x \neq 1 \land y \neq 1\right\}\;.
    \end{align*}
The resulting convergence history for first- and second-order approximations -- the system was solved applying a primal-dual active set strategy and adaptive mesh refinement using a Dörfler marking strategy (bulk parameter of $0.5$) -- is reported in Table \ref{tab:PoissonReLu-LShape}. Optimal convergence rates are achieved. The efficiency of the error estimate is very close to one.
\begin{table}[h]
\centering
\begin{tabular}{@{}ll|llllll@{}}
\toprule
$k$ & $n$ & $n_{DOF}$ & $\vert\vert\vert \textbf{u}-\textbf{u}_\cT \vert\vert\vert$ & & $\mathcal{G}^{1/2}$ & & $i_{eff}$ \\ \midrule
1 & 1  & $81$    & $8.79e-1$ & $-$    & $8.71e-1$ & $-$    & $0.99$ \\
1 & 2  & $131$   & $6.64e-1$ & $0.58$ & $6.62e-1$ & $0.59$ & $1.00$ \\
1 & 3  & $232$   & $5.16e-1$ & $0.44$ & $5.15e-1$ & $0.46$ & $1.00$ \\
1 & 4  & $388$   & $3.84e-1$ & $0.57$ & $3.84e-1$ & $0.54$ & $1.00$ \\
1 & 9  & $8389$  & $7.85e-2$ & $0.51$ & $7.85e-2$ & $0.52$ & $1.00$ \\
1 & 10 & $15566$ & $5.75e-2$ & $0.50$ & $5.75e-2$ & $0.50$ & $1.00$ \\ \midrule
2 & 1  & $257$   & $3.53e-1$ & $-$    & $3.53e-1$ & $-$    & $1.00$ \\
2 & 2  & $387$   & $2.25e-1$ & $1.10$ & $2.24e-1$ & $1.10$ & $1.00$ \\
2 & 3  & $531$   & $1.46e-1$ & $1.37$ & $1.46e-1$ & $1.36$ & $1.00$ \\
2 & 4  & $754$   & $9.79e-2$ & $1.14$ & $9.78e-2$ & $1.14$ & $1.00$ \\
2 & 9  & $2930$  & $1.93e-2$ & $1.10$ & $1.93e-2$ & $1.11$ & $1.00$ \\
2 & 10 & $4279$  & $1.25e-2$ & $1.16$ & $1.25e-2$ & $1.15$ & $1.00$ \\ \bottomrule
\end{tabular}
\caption{Convergence history of the ReLU-type nonlinearity on an L-shaped domain: Values for the error $\vert\vert\vert \mathbf{u} - \mathbf{u}_\cT \vert\vert\vert$ following \eqref{eq:error_norm_poisson}, the Least-Squares functional $\mathcal{G}$ and the effectivity index effectivity index $i_{eff}$ using $u_\cT \in \left(\mathrm{P}_k\right)^2$ and $\sigma_\cT \in \left(\mathrm{RT}_{k-1}\right)^2$ with $k\in\{1,\, 2\}$.}
\label{tab:PoissonReLu-LShape}
\end{table}

% --- Nonlinear elasticity
\subsection{Nonlinear elasticity}
A fundamental tool in structural analysis is the finite element method, the basis of many mechanical simulations in engineering practice. In most applications, materials are modeled using elastic or elastoplastic constitutive laws, often assuming small deformations or, more general, large deformations at small strains. While small-deformation elasticity leads to a linear system of equations, the transition to the large-deformation setting introduces nonlinearities. In this section, large deformations at small strains are considered, for which the Saint Venant–Kirchhoff model provides a prototypical formulation. Introducing deformation gradient, Green-Lagrange and engineering strain
\begin{equation*}
    F(u) = I + \nabla u,
    \quad
    E(u) = \frac{1}{2}\left(F^\mathrm{T}(u)F(u) - I\right)
    \quad\text{respectively}\quad
    \varepsilon(u) = \mathrm{sym}\,\nabla u\; ,
\end{equation*}
the balance of linear momentum in its first-order form reads
\begin{equation}
\div P = -f \quad\text{with}\quad P(u) = F(u)\,\left(2\,E(u) + \lambda\,\mathrm{tr}E(u)\,I\right)\; .
\end{equation}
The first Piola–Kirchhoff stress $P$ maps reference-area forces to the current configuration. Classical results by Ciarlet \cite{C88} guarantee existence and uniqueness of a solution 
\begin{equation*}
    (u,\;P) \in \left(\left(W^{2,p}(\Omega)\right)^d \cap \left(W^{1,p}_0(\Omega)\right)^d\right) \times \left(W^{1,p}(\Omega)\right)^{d\times d}\; ,
\end{equation*}
for sufficiently small data $f \in (L^p(\Omega))^d$ and $p>3$. 
In the notation of this contributions, the corresponding Least-Squares functional reads
\begin{equation}
    \mathcal{M}(u,\sigma;f) = \omega\,\Vert\div\sigma + f\Vert^2 + \Vert \sigma - F\,\left(2\,E(u) + \lambda\,\mathrm{tr}E(u)\,I\right)\Vert \; ,
    \label{eq:ls_elasticity_general}
\end{equation}
with a weight $\omega \in \mathbb{R}_{>0}$. Conformity with the abstract functional \eqref{eq:ls_general} is given for the norms $\| \cdot \|_{H_1} = \omega\,\| \cdot \|$ and $\| \cdot \|_{H_2}=\| \cdot \|$, the solution space $\bfV = W^{1,\infty}_{\Gamma_D}(\Omega)^{d} \times W^\infty_{\Gamma_N}(\div;\Omega)^{d}$ (compare \cite{MSSS14}) and the linear- respectively nonlinear contributions
\begin{equation}
    \begin{split}
        L_1(u, \sigma) &= \div \sigma\\
        L_2(u, \sigma) &= \sigma - 2\varepsilon(u) - \lambda\mathrm{tr}\,\varepsilon(u)\,I\\
        K(u) &= \left[2\,\nabla^\mathrm{T} u
        + \nabla u
        + \nabla u \nabla^\mathrm{T} u 
        + \lambda\,\mathrm{tr}\,E(u)\right]\,\nabla u
        + \frac{\lambda}{2}\,\mathrm{tr}\left(\nabla^\mathrm{T}u\nabla u\right)\,I\; .
    \end{split}
    \label{eq:nlelast_definitions_framework}
\end{equation}

\noindent Similar to the nonlinear heat-equation in section \ref{sec:nlin_poisson}, Assumptions \ref{ass_linear} and \ref{ass_lipschitz} have to be verified. Starting with the coercivity of the linear part of the Least-Squares functional \eqref{eq:ls_elasticity_general}, $\omega\,\Vert L_1(u,\sigma)\Vert^2 + \Vert L_2(u, \sigma)  \Vert^2$, Korn's inequality \cite{BS08}
\begin{equation}
 \Vert v\Vert_1 \leq C_K\,\Vert\varepsilon(v)\Vert\; ,
 \label{eq:korns_inequality}
\end{equation}
adjusted to the given boundary conditions is required.
\begin{theorem}
For any $C_{1,\mathcal{M}} \in (0,\;4\left(1+C_K\right)^{-2})$, there exist a constant $C_{2,\mathcal{M}} > 0$, depending only on $\Omega$, $\lambda$ and $\omega$, such that the linear part of the Least-Squares functional $\mathcal{M}$ is coercive, i.e.
\begin{equation*}
     C_{1,\mathcal{M}}\,\Vert \varepsilon(u)\Vert + C_{2,\mathcal{M}}\,\Vert \sigma\Vert_{\div} \leq \omega\,\Vert L_1(u,\sigma)\Vert^2 + \Vert L_2(u, \sigma)\Vert^2
\end{equation*}
for all $(u,\sigma) \in H_{\Gamma_D}^1(\Omega) \times H_{\Gamma_N}(\div; \Omega).$
\end{theorem}
\begin{proof}
    The definition of $L_2$ in \eqref{eq:nlelast_definitions_framework} yields
        \begin{equation}
            \begin{split}
                4\,\Vert \varepsilon(u) \Vert^2 
                &= \left(2\,\varepsilon(u),\; 2\,\varepsilon(u) + \lambda\,\mathrm{tr}\,\varepsilon(u)\,I - \sigma\right) + 2\,\left(\varepsilon(u),\; \sigma\right) - 2\,\left(\varepsilon(u),\; \lambda\,\mathrm{tr}\,\varepsilon(u)\,I\right)\\
                &\leq 2\,\Vert \varepsilon (u) \Vert\,\Vert L_2(u,\sigma) \Vert + 2 \,\left(\varepsilon(u),\; \sigma\right)\; 
            \end{split}
            \label{eq:a1_elasticity_norm_eps}
        \end{equation}
    while an integration by parts combined with Cauchy-Schwarz and Korn inequalities results in
        \begin{equation*}
            \vert\left(\varepsilon(u),\; \sigma\right)\vert
            = \vert\left(u,\; \div\sigma\right) + \left(\mathrm{as}\,\nabla u,\; L_2(u,\sigma)\right)\vert
            \leq \left(\Vert L_1(u,\sigma)\Vert + \Vert L_2(u,\sigma)\Vert \right)\,\Vert u \Vert_1
            \leq C_K\ \left(\Vert L_1(u,\sigma)\Vert + \Vert L_2(u,\sigma)\Vert \right)\,\Vert \varepsilon(u)\Vert\; .
        \end{equation*}
     Inserting 
        \begin{equation*}
            \Vert L_1(u,\sigma)\Vert \leq \omega^{-1/2}\left(\omega\,\Vert L_1(u,\sigma)\Vert^2 + \Vert L_2(u, \sigma)  \Vert^2\right)^{1/2} \quad\text{and}\quad \Vert L_2(u,\sigma)\Vert \leq \left(\omega\,\Vert L_1(u,\sigma)\Vert^2 + \Vert L_2(u, \sigma)\Vert^2\right)^{1/2}
        \end{equation*}
   into \eqref{eq:a1_elasticity_norm_eps} leads to
        \begin{equation}
            \|\varepsilon(u)\| \leq \frac{1 + C_K\left(1+\omega^{-1/2}\right)}{2}\,\left(\omega\,\Vert L_1(u,\sigma)\Vert^2 + \Vert L_2(u, \sigma)\Vert^2\right)^{1/2}\; .
            \label{eq:elasticity_upper-bound-epsilon}
        \end{equation}      
    Moreover, using $\Vert\sigma\Vert \leq \Vert L_2(u,\sigma)\Vert + \Vert 2\varepsilon(u) + \lambda\mathrm{tr}\,\varepsilon(u)\,I\Vert$ alongside with the bound on $\Vert\varepsilon(u)\Vert$ yields
        \begin{equation}
            \Vert\sigma\Vert \leq \Vert L_2(u,\sigma)\Vert + (2+d\lambda)\,\Vert\varepsilon(u)\Vert 
            \leq \left(1+ \left(2+d\lambda\right) \,\frac{1+C_K(1+\omega^{-1/2})}{2}\right)\,\left(\omega\,\Vert L_1(u,\sigma)\Vert^2 + \Vert L_2(u, \sigma)\Vert^2\right)^{1/2}\; .
            \label{eq:elasticity_upper-bound-sigma}
        \end{equation}
    Adding up the upper bounds \eqref{eq:elasticity_upper-bound-epsilon} and \eqref{eq:elasticity_upper-bound-sigma}, proves coercivity of the linear part of $\mathcal{M}$ with
    \begin{equation*}
        C_{1,\mathcal{M}} =\frac{4}{\left(1+C_K(1+\omega^{-1/2})\right)^{2}}
        \quad\text{and}\quad
        C_{2,\mathcal{M}} = \min\!\left\{\omega,\ \frac{1}{\left(1+\left(2+d\lambda\right)\,\frac{1+C_K(1+\omega^{-1/2})}{2}\right)^{2}}\right\}\; .
    \end{equation*}
    Since $C_{1,\mathcal{M}}$ is strictly increasing within
        \begin{equation*}
            0\ < C_{1,\mathcal{M}} <\ \frac{4}{\left(1+C_K\right)^2}\; ,
        \end{equation*}
    the proposition holds.
\end{proof}

\begin{theorem}
There exist a constant $\rho > 0$, depending on $d$, $\lambda$, and $C_K$, such that for all 
\begin{equation*}
    u,v \in U_\rho :=\left\{\,w\in W^{1,\infty}_{\Gamma_D}(\Omega)^d:\ \left\|\nabla w\right\|_{L^\infty(\Omega)} \leq \rho\,\right\}\; ,
\end{equation*}
the nonlinearity $K$ following \eqref{eq:nlelast_definitions_framework} is Lipschitz continuous
\begin{equation*}
    \Vert K(u)-K(v)\Vert^2  \leq L\,\Vert\varepsilon(u-v)\Vert^2\; ,
\end{equation*}
with a constant $L\in\left(0,\,4/\left(1+C_K\right)^{2}\right)$.
\end{theorem}
\begin{proof}
The nonlinearity $K(w)$ is decomposed into a sum of simpler tensorial contributions
\begin{equation*}
    K(w) = \sum_{i=1}^4 k_i(w) \quad\text{with}\quad 
    k_1(w) = 3\,\nabla^\mathrm{T}w\,\nabla w,\quad
    k_2(w) = \nabla w\,\nabla^\mathrm{T}w\,\nabla w,\quad
    k_3(w) = \lambda\,\mathrm{tr}\,E(w)\,\nabla w
    \quad\text{and}\quad
    k_4(w) = \dfrac{\lambda}{6}\,\mathrm{tr}\,k_1(w)\,I\; ,
\end{equation*}
where each can be estimated separately. Applying Hölders- and Korn's inequality, yields upper bounds for the first two contributions:
\begin{align*}
    \Vert k_1(u)-k_1(v)\Vert
        &= 3\,\Vert \nabla^\mathrm{T}u\,\nabla u - \nabla^\mathrm{T}v\,\nabla v\Vert
        = 3\,\Vert (\nabla(u-v))^\mathrm{T}\nabla u + (\nabla v)^\mathrm{T}\nabla(u-v)\Vert \\
        &\leq 3\left(\Vert \nabla u\Vert_{L^\infty(\Omega)}+\Vert\nabla v\Vert_{L^\infty(\Omega)}\right)\,\Vert\nabla(u-v)\Vert
        \leq 3C_K\,\Big(\Vert \nabla u\Vert_{L^\infty(\Omega)}+\Vert\nabla v\Vert_{L^\infty(\Omega)}\Big)\,\Vert\varepsilon(u-v)\Vert\\
     \Vert k_2(u)-k_2(v)\Vert
        &= \Vert \nabla u\,\nabla^\mathrm{T}u\,\nabla u - \nabla v\,\nabla^\mathrm{T}v\,\nabla v\Vert \\
        &\leq \left(\Vert \nabla u\Vert_{L^\infty(\Omega)}^2+\Vert \nabla u\Vert_{L^\infty(\Omega)}\Vert \nabla v\Vert_{L^\infty(\Omega)}+\Vert \nabla v\Vert_{L^\infty(\Omega)}^2\right)\,\Vert \nabla(u-v)\Vert
        \leq C_K\,\Big(\Vert \nabla u\Vert_{L^\infty(\Omega)} + \Vert \nabla v\Vert_{L^\infty(\Omega)}\Big)^2\,\Vert\varepsilon(u-v)\Vert \;.
\end{align*}
In order to bound $k_3$, a reformulation of the Green-Lagrange strain and the following upper bound on the norm of the trace operator
\begin{equation*}
    \mathrm{tr}\,E(w) = \div w + \frac{1}{2}\,\mathrm{tr}(\nabla^{\mathrm T}w\,\nabla w)
    \quad\text{respectively}\quad
    \vert \mathrm{tr}\,M| \leq \sqrt{d}\,\vert M\vert
\,\end{equation*}
are required. Adding and subtracting $E(v)\,\nabla u$, using these relations and applying Korn's inequality \eqref{eq:korns_inequality} yields 
\begin{align*}
    \Vert k_3(u)-k_3(v)\Vert 
    &= \lambda\,\Vert \left[\mathrm{tr}\,E(u) - \mathrm{tr}\,E(v)\right]\,\nabla u + \mathrm{tr}\,E(v)\,\nabla\left(u-v\right) \Vert \\ 
    &\leq \lambda\,\Big(\Vert \mathrm{tr}\,E(u)-\mathrm{tr}\,E(v)\Vert\,\Vert \nabla u\Vert_{L^\infty(\Omega)} + \Vert \mathrm{tr}\,E(v)\Vert_{L^\infty(\Omega)}\,\Vert\nabla(u-v)\Vert\Big)\\
    &\leq \sqrt{d}\,\lambda\,\left(\|\nabla u\|_{L^\infty(\Omega)}+\|\nabla v\|_{L^\infty(\Omega)} + \frac{1}{2}\,\big(\|\nabla u\|_{L^\infty(\Omega)}^2+\|\nabla u\|_{L^\infty(\Omega)}\|\nabla v\|_{L^\infty(\Omega)}+\|\nabla v\|_{L^\infty(\Omega)}^2\big)\right)\,\|\nabla(u-v)\|\\
    &\leq C_K\,\sqrt{d}\,\lambda\,\left(\|\nabla u\|_{L^\infty(\Omega)}+\|\nabla v\|_{L^\infty(\Omega)} + \frac{1}{2}\,\big(\|\nabla u\|_{L^\infty(\Omega)}+\|\nabla v\|_{L^\infty(\Omega)}\big)^2\right)\,\|\varepsilon(u-v)\|\; .
\end{align*}

\noindent Applying once again the upper bound of the trace operator, alongside with Korns inequality, bounds $k_4$ by
\begin{equation*}
    \Vert k_4(u)-k_4(v)\Vert \leq C_K\,\frac{\sqrt{d}\,\lambda}{2}\,\Big(\Vert \nabla u\Vert_{L^\infty(\Omega)}+\Vert \nabla v\Vert_{L^\infty(\Omega)}\Big)\,\Vert \varepsilon(u-v)\Vert \; .
\end{equation*}

\noindent Summing the above estimates yields
\begin{equation*}
\Vert K(u)-K(v)\Vert \leq C_K\left(
\left(3+\lambda\,\left(1+\frac{\sqrt d}{2}\right)\right)\,
\Big(\|\nabla u\|_{L^\infty(\Omega)}+\|\nabla v\|_{L^\infty(\Omega)}\Big)
+ \left(1+\frac{\sqrt{d}\,\lambda}{2}\right)\,
\Big(\|\nabla u\|_{L^\infty(\Omega)}+\|\nabla v\|_{L^\infty(\Omega)}\Big)^2
\right)\,\|\varepsilon(u-v)\|.
\end{equation*}
As $\Vert \nabla u\Vert_{L^\infty(\Omega)}$ respectively $\Vert\nabla v\Vert_{L^\infty(\Omega)}$ can be bounded by $\rho$, the Lipschitz constant $L$ of $K$ is bounded by
\begin{equation*}
    \sqrt{L(\rho)} \leq C_K\,\left(2\,\left(3 + \lambda\,\left(1 + \frac{\sqrt d}{2}\right)\right)\,\rho + 4\,\left(1 + \frac{\sqrt{d}\,\lambda}{2}\right)\,\rho^2\right)\; .
\end{equation*}

\noindent Choosing
\begin{equation*}
    \rho < \frac{-\left(6 + 2\lambda + \lambda\sqrt{2}\right) + \sqrt{
    \left(6 + 2\lambda + \lambda\sqrt{2}\right)^2 + \frac{8\left(4 + 2\sqrt{2}\lambda\right)}{C_K(1 + C_K)}}}{2\left(4 + 2\sqrt{2}\lambda\right)}
\end{equation*}
ensures $L(\rho) < 4/(1+C_K)^2$ and consequently 
\[
\|K(u)-K(v)\| \leq L(\rho)\,\|\varepsilon(u-v)\|,
\]
for all $u,v\in U_\rho$.
\end{proof}

Consequently, Theorems \ref{thm:normequivalence} and \ref{thm:aposteriori_general} apply, which means that the Least-Squares functional \eqref{eq:ls_elasticity_general} serves as an error indicator. In two numerical examples, a squared domain under uniform refinement, and the Cooks membrane (see e.g. \cite{S21}) under adaptive mesh refinement are considered. Convergence in this section is measured within the norm
\begin{equation}
    \vert\vert\vert \textbf{u}-\textbf{u}_\cT \vert\vert\vert^2  = \|\varepsilon\left(u-u_\cT\right)\|^{2} + \Vert\sigma-\sigma_\cT\Vert_{\div}^2\; ,
    \label{eq:error_norm_elasticity}
\end{equation}
while second-order discretizations  $u_\cT \in \left(\mathrm{P}_2\right)^2$ and $\sigma_\cT \in \left(\mathrm{RT}_1\right)^2$ are applied.

\textbf{Example 1:}\\
A unit square with a manufactured solution
\begin{equation*}
    u = \frac{7}{100}
    \begin{pmatrix}
         \sin(\pi\,x)\;\sin(\pi\,y)\\
        \sin(\pi\,x)\;\sin(2\pi\,y)
    \end{pmatrix}
    \quad\text{using}\quad f = -\div \sigma(u)
\end{equation*}
is considered.
Homogeneous Dirichlet boundary conditions are applied on the entire boundary $\partial\Omega$.
The convergence history on a sequence of uniformly refined meshes is reported in Table \ref{tab:ElasticityNL-ManSol}.
\begin{table}[h]
\centering
\begin{tabular}{@{}l|lllll@{}}
\toprule
$h$ & $\vert\vert\vert \mathbf{u} - \mathbf{u}_\cT \vert\vert\vert$ & & $\mathcal{M}^{1/2}$ & & $i_{eff}$ \\ \midrule
1/2  & $2.49e+0$ & $-$    & $2.46e+0$ & $-$    & $0.99$ \\
1/4  & $1.07e+0$ & $1.21$ & $1.07e+0$ & $1.19$ & $1.00$ \\
1/8  & $2.48e-1$ & $2.11$ & $2.49e-1$ & $2.11$ & $1.00$ \\  
1/16 & $6.38e-2$ & $1.96$ & $6.39e-2$ & $1.96$ & $1.00$ \\
1/32 & $1.61e-2$ & $1.99$ & $1.61e-2$ & $1.99$ & $1.00$ \\         
1/64 & $4.02e-3$ & $2.00$ & $4.03e-3$ & $2.00$ & $1.00$ \\   \bottomrule
\end{tabular}
\caption{Convergence history for a manufactured solution of nonlinear elasticity: Values for the error $\vert\vert\vert \mathbf{u} - \mathbf{u}_\cT \vert\vert\vert$ following \eqref{eq:error_norm_elasticity}, the Least-Squares functional $\mathcal{M}$ and the effectivity index $i_{eff}$ using $u_\cT \in \left(\mathrm{P}_2\right)^2$, $\sigma_\cT \in \left(\mathrm{RT}_1\right)^2$ and $\omega = 1$.}
\label{tab:ElasticityNL-ManSol}
\end{table}
Optimal e.o.c. for the error as well as the Least-Squares functional are achieved. The efficiency of the error estimate is close to one, the Gauss-Newton procedure requires between 11 and 14 iterations to converge.

\textbf{Example 2:}\\
In order to demonstrate the usage of the Least-Squares functional within an adaptive solution procedures, the Cook's membrane is considered. 
The geometry with initial discretisation is shown in Figure \ref{fig:ElasticityNL-Cook} (a).
To improve convergence, a scaling factor $\omega = 10^3$ is considered.
Meshes are refined using a Dörfler marking strategy with bulk parameter $0.5$, leading to a considerable clustering of elements in the upper left and right as well as the lower right corners of the domain, as depicted in Figure \ref{fig:ElasticityNL-Cook} (b).
\begin{figure}[h]
    \centering
    \begin{subfigure}{0.45\textwidth}
        \centering
        \begingroup
        \definecolor{isdred}{HTML}{D90429}
        \definecolor{isdgreen}{HTML}{5B8C5A}
        
        \begin{tikzpicture}[]
            % --- The mesh
            \node[inner sep=0pt] at (0, 0) {\includegraphics[scale=0.2]{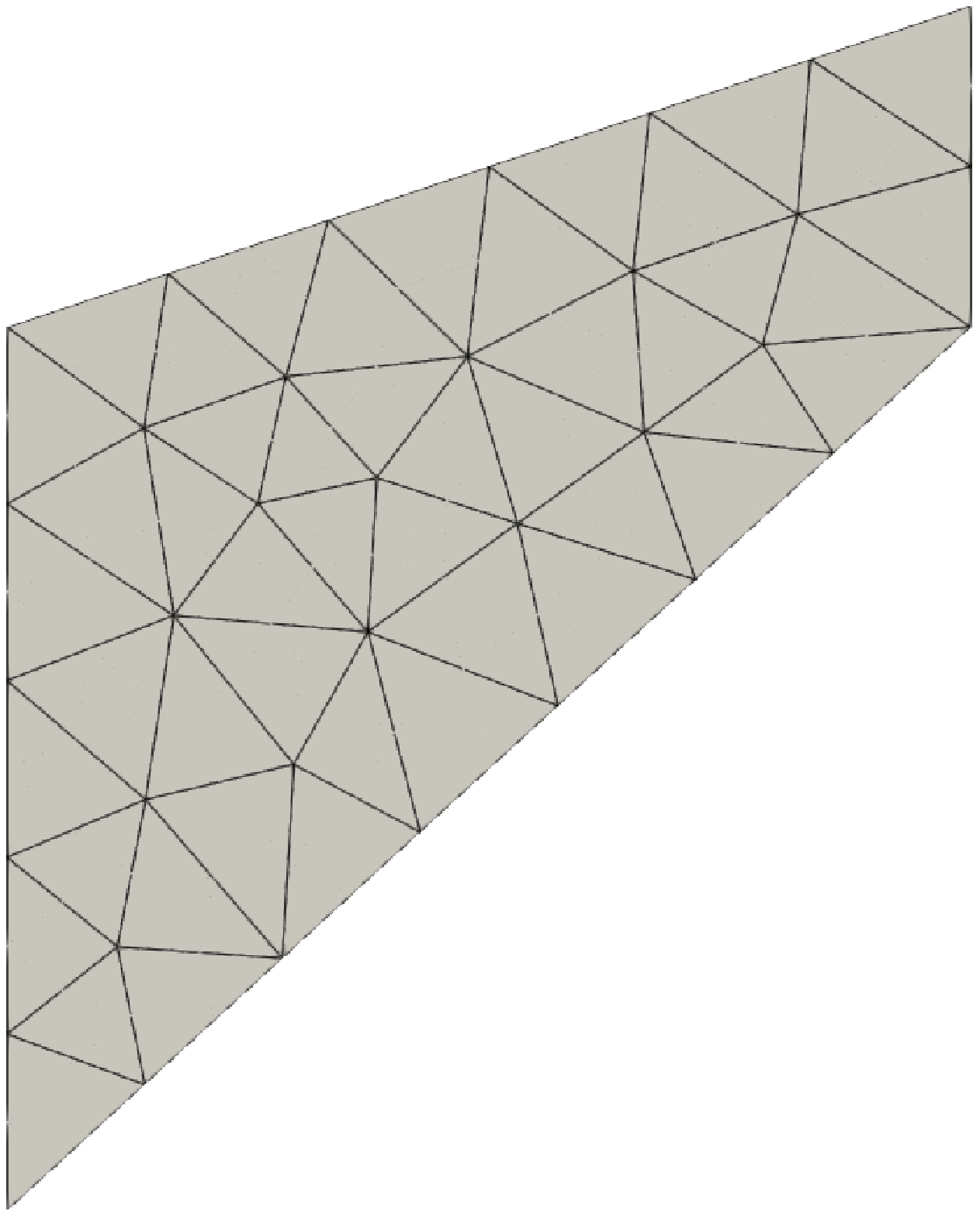}};
        
            % --- The scetch
            % Points
            \coordinate (A) at (-2.1, -2.65);
            \coordinate (B) at (2.1, 1.2);
            \coordinate (C) at (2.1, 2.63);
            \coordinate (D) at (-2.1, 1.2);
        
            % Curves
            \draw[line width=1pt] (A) -- (B) -- (C) -- (D)-- (A);
        
            % Bearing
            \foreach \i in {0,1,...,10} {
                \draw[line width=1pt] (-2.1, -2.65 + \i*0.385) -- (-2.25, -2.55 + \i*0.385);
            }
        
            % Loads
            \draw[->, color=isdred] (2.2, 1.2) -- (2.2, 1.65);
            \draw[->, color=isdred] (2.2, 1.68) -- (2.2, 2.13);
            \draw[->, color=isdred] (2.2, 2.15) -- (2.2, 2.65);
            \node[] at (2.4, 1.915){\textcolor{isdred}{$t$}};
        
            % Dimensions
            \draw[|-|, color=isdgreen] (-2.1, -2.85) -- (2.1, -2.85);
            \node[] at (-1.22, -1.72){ };
            \draw[|-|, color=isdgreen] (2.7, -2.65) -- (2.7, 1.2);
            \draw[-|, color=isdgreen] (2.7, 1.2) -- (2.7, 2.63);
        
            \node[] at (0, -3.15){\textcolor{isdgreen}{$48$}};
            \node[] at (3, 1.915){\textcolor{isdgreen}{$16$}};
            \node[] at (3, -0.725){\textcolor{isdgreen}{$44$}};
        \end{tikzpicture}
        \endgroup
        \caption{}
    \end{subfigure}
    \begin{subfigure}{0.45\textwidth}
        \centering
        \begin{tikzpicture}[]
            \node[inner sep=0pt] at (0, 0) {\includegraphics[scale=0.2]{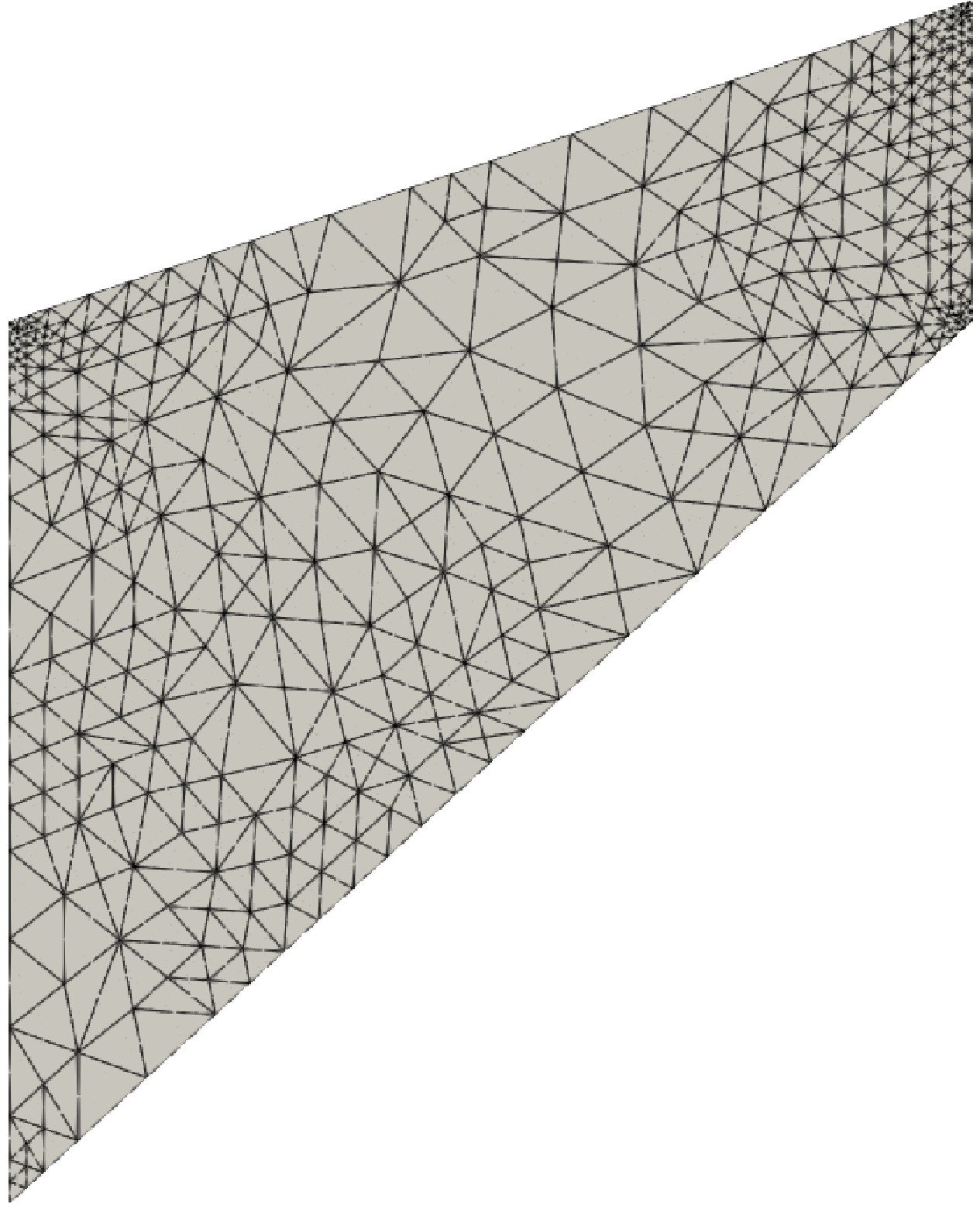}};
            \node[] at (0, -3.15){\textcolor{white}{$48$}};
        \end{tikzpicture}
        \caption{}
    \end{subfigure}
    \caption{Cooks membrane: The boundary value problem with initial mesh (a) and the mesh after eight adaptive refinement cycles (b).}
    \label{fig:ElasticityNL-Cook}
\end{figure}
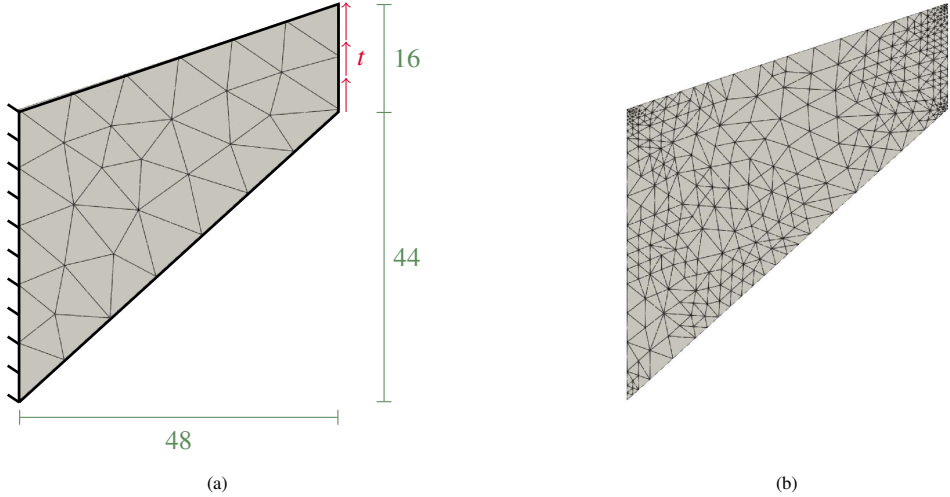

\noindent The convergence history of this experiment is detailed in Table \ref{tab:ElasticityNL-Cook}.
Errors are calculated based on an overkill solution on a refined grid using $u \in \left(\mathrm{P}_3\right)^2$ and $\sigma \in \left(\mathrm{RT}_2\right)^2$.
\begin{table}[ht]
\centering
\begin{tabular}{@{}ll|lllll@{}}
\toprule
$n$ & $n_{DOF}$ & $\vert\vert\vert \mathbf{u} - \mathbf{u}_\cT \vert\vert\vert$ & & $\mathcal{M}^{1/2}$ & & $i_{eff}$ \\ \midrule
1 & $754$   & $2.93e-1$ & $-$    & $1.31e-1$ & $-$    & $0.45$ \\
2 & $948$   & $2.38e-1$ & $0.92$ & $1.04e-1$ & $1.03$ & $0.44$ \\
3 & $1188$  & $1.79e-1$ & $1.24$ & $8.58e-2$ & $0.84$ & $0.48$ \\  
4 & $1614$  & $1.03e-3$ & $1.81$ & $6.16e-2$ & $1.08$ & $0.60$ \\
7 & $6870$  & $2.26e-2$ & $0.90$ & $1.59e-2$ & $0.96$ & $0.70$ \\         
8 & $11352$ & $1.34e-2$ & $1.04$ & $9.74e-3$ & $0.97$ & $0.73$ \\   \bottomrule
\end{tabular}
\caption{Convergence history for the Cooks membrane: Values for the error $\vert\vert\vert \mathbf{u} - \mathbf{u}_\cT \vert\vert\vert$ following \eqref{eq:error_norm_elasticity}, the Least-Squares functional $\mathcal{M}$ and the effectivity index $i_{eff}$.}
\label{tab:ElasticityNL-Cook}
\end{table}
As expected, an optimal e.o.c is achieved, while between six and eight iterations of the Gauss-Newton solver are required.

% --- Sea-ice dynamics
\subsection{Sea-Ice}
The final example examines a model of sea-ice dynamics governed by the momentum balance \cite{seaice, Hibler1979}
\begin{equation}\label{eq:momentum}
\rho h\, u_t = \nabla \cdot \sigma - \tau_o(u-v_o) + F(u).
\end{equation}
Therein, $u$ is the velocity of the sea-ice, $\sigma$ the stress tensor, $v_o$ the water velocity, $F(u)$ the external forces and the drag of the ocean current
\begin{equation*}
    \tau_o(u) = \rho_oC_o|u|u\; .
\end{equation*}
It additionally depends on the water drag coefficient $C_o$ and the sea-water density $\rho_o$.
With an appropriate scaling and the introduction of two parameters $\beta, \eta>0$, the Least-Squares functional is given by 
\begin{align*}
    \mathcal{I}(u,\sigma; v_o) :=
    \Vert 
    \beta^{1/2} (u - u_{\mathrm{old}}) +
     \beta^{-1/2}\tau_o(u) -  \beta^{-1/2}\div\sigma \Vert^2
    + \Vert   ({2\eta})^{-1/2}
    \sigma -  ({2\eta})^{1/2}\varepsilon(u) + ({2\eta})^{1/2}\varepsilon(v_o)  \Vert^2\; .
    \label{eq:functionalF}
\end{align*}
Conformity with the abstract functional \eqref{eq:ls_general} is achieved, using the norms $\| \cdot \|_{H_1} = \| \cdot \|_{H_2}=\| \cdot \|$, the solution space $\bfV = H^{1}_{\Gamma_D}(\Omega)^{d} \times H_{\Gamma_N}(\div;\Omega)^{d}$, as well as the definitions
\begin{equation*}
    L_1(u, \sigma) = ({2\eta})^{-1/2} \sigma -  ({2\eta})^{1/2}\varepsilon(u),\quad
    L_2(u, \sigma) = \beta^{-1/2} u -  \beta^{1/2}\div\sigma\quad\text{and}\quad
    K(u) =  \beta^{1/2}\tau_o (u)\; .
\end{equation*}
Since the local Lipschitz constant of $\tau_o$ becomes small in a neighborhood of zero in $W^{1,\infty}$, the analysis reduces to the coercivity and continuity of the linear functional 
\begin{align}
    \Vert L_1(u,\sigma)\Vert^2 + \Vert L_2(u, \sigma)\Vert^2
\end{align}
 for in $(u,\sigma)\in\left(H_{\Gamma_D}^1(\Omega)\right)^d \times \left(H_{\Gamma_N}(\div; \Omega)\right)^d$, shown in \cite{seaice}. 
 Note that $K(u) \in L^{2}$, since $H^{1} \hookrightarrow L^{6}$.
 
 \textbf{Example 1:}\\
 A unit square with a manufactured solution 
 \begin{align*}
     u = \begin{pmatrix}
         \sin(\pi x) \sin (\pi y) \\ \sin(\pi x) \sin(2\pi y)
     \end{pmatrix}
 \end{align*}
 with $v_o = 2 u$ and $u_{\mathrm{old}}$ chosen appropriately, is considered. 
 Homogeneous Dirichlet boundary conditions are applied throughout the boundary 
 $\partial \Omega$. For simplicity, we chose $\rho_oC_o=1$, $\beta =1$ and $\eta =2$. The convergence history for $u_\cT \in \left(\mathrm{P}_2\right)^2$ and $\sigma_\cT \in \left(\mathrm{RT}_1\right)^2$ on a sequence of uniformly refined meshes is reported in Table \ref{tab:sea_ice}. 
 Optimal convergence rates of the error as well as the Least-Squares functional are achieved. The efficiency of the estimator is around 
 $0.2$ and the Gauss-Newton method converges within seven steps.
\begin{table}[h]
\centering
\begin{tabular}{@{}l|lllll@{}}
\toprule
$h$ & $\Vert \nabla(u-u_\cT)  \Vert$ & & $\mathcal{I}^{1/2}$ & & $i_{eff}$ \\ \midrule
1/2  & $3.04e+0$ & $-$ & $1.10e+0$ & $-$    & $0.36$ \\
1/4  & $1.76e+0$ & $0.79$ & $4.56e-1$ & $1.27$ & $0.26$ \\
1/8  & $6.78e-1$ & $1.38$ & $1.35e-1$ & $1.75$ & $0.20$ \\  
1/16 & $1.95e-1$ & $1.80$ & $3.89e-2$ & $1.80$ & $0.20$ \\
1/32 & $5.11e-2$ & $1.93$ & $1.02e-2$ & $1.93$ & $0.20$ \\   
1/64 & $1.03e-2$ & $2.31$ & $2.60e-3$ & $1.97$ & $0.25$ \\   \bottomrule
\end{tabular}
\caption{Convergence history for a manufactured solution of the sea ice problem: Values for the error $\Vert \nabla\left(\mathbf{u} - \mathbf{u}_\cT\right)\Vert$ following \eqref{eq:error_norm_elasticity}, the Least-Squares functional $\mathcal{I}$ and the effectivity index $i_{eff}$ using $u_\cT \in \left(\mathrm{P}_2\right)^2$ and $\sigma_\cT \in \left(\mathrm{RT}_1\right)^2$.}
\label{tab:sea_ice}
\end{table}
\section*{Funding}
The fourth author gratefully acknowledges support by DFG in the Priority Programme SPP 962 "Non-smooth and Complementarity-based Distributed Parameter Systems" under grant number STA 402/13-2 

\bibliographystyle{plain}
\bibliography{bib.bib}
\end{document}